\documentclass[onecolumn,hidelinks,onefignum,onetabnum]{siamart250211}

\usepackage{lipsum}
\usepackage{amsfonts}
\usepackage{graphicx}
\usepackage{epstopdf}
\usepackage{algorithm}
\usepackage{algpseudocode}
\ifpdf
  \DeclareGraphicsExtensions{.eps,.pdf,.png,.jpg}
\else
  \DeclareGraphicsExtensions{.eps}
\fi

\newsiamremark{remark}{Remark}
\newsiamremark{hypothesis}{Hypothesis}
\crefname{hypothesis}{Hypothesis}{Hypotheses}
\newsiamthm{claim}{Claim}
\newsiamremark{fact}{Fact}
\crefname{fact}{Fact}{Facts}

\headers{Adaptive Inertial Method}{Han Long}

\title{Adaptive Inertial Method\thanks{Submitted to the editors DATE.
\funding{This work was funded by Theme-based Research Project T32-615/24-R from the Hong Kong Research Grants Council.}}}

\author{Han Long\thanks{Department of Industrial Engineering and Decision Analytics, The Hong Kong University of Science and Technology, Hong Kong SAR, China (\email{han.long@connect.ust.hk}, \email{jiheng@ust.hk}).}
\and Bingsheng He\thanks{Department of Mathematics, Nanjing University, Nanjing, China (\email{hebma@nju.edu.cn}).}
\and Yinyu Ye\thanks{Department of Management Science and Engineering, Stanford University, Stanford, CA, USA, and Department of Industrial Engineering and Decision Analytics, The Hong Kong University of Science and Technology, Hong Kong SAR, China (\email{yyye@stanford.edu}).}
\and Jiheng Zhang\footnotemark[2]
}

\usepackage{amsopn}
\DeclareMathOperator{\Diag}{Diag}

\usepackage{bm, tikz, mathtools, booktabs}
\DeclareMathOperator{\argmin}{arg\,min}
\usetikzlibrary{arrows.meta, angles}

\usepackage{rotating, subcaption, multirow}

\ifpdf
\hypersetup{
  pdftitle={Adaptive Inertial Method},
  pdfauthor={Han Long}
}
\fi

\begin{document}

\maketitle

\begin{abstract} 
	In this paper, we introduce the Adaptive Inertial Method (AIM), a novel framework for accelerated first-order methods through a customizable inertial term.
	We provide a rigorous convergence analysis establishing a global convergence rate of $O(1/k)$ under mild conditions, requiring only convexity and local Lipschitz differentiability of the objective function.
	Our method enables adaptive parameter selection for the inertial term without manual tuning.
	Furthermore, we derive the particular form of the inertial term that transforms AIM into a new Quasi-Newton method.
	Notably, under specific circumstances, AIM coincides with the regularized Newton method, achieving an accelerated rate of $O(1/k^2)$ without Hessian inversions.
	Through extensive numerical experiments, we demonstrate that AIM exhibits superior performance across diverse optimization problems, highlighting its practical effectiveness.
\end{abstract}



\section{Introduction}

\par Consider the unconstrained convex optimization problem:
\begin{equation}
	\min\limits_{\bm{x} \in \mathbb{R}^n} f(\bm{x}),
	\label{prob_uc}
\end{equation}
where $f:\mathbb{R}^n \rightarrow \mathbb{R}$ is convex and differentiable with Lipschitz continuous gradient, i.e., there exists a constant $D > 0$ such that
\begin{equation*}
	\|\nabla f(\bm{x}) - \nabla f(\bm{y})\| \leq D \|\bm{x} - \bm{y}\|, \quad \forall \bm{x},\ \bm{y} \in \mathbb{R}^n.
\end{equation*}

\par Accelerated first-order methods provide more efficient approaches to solving problem \eqref{prob_uc} than standard gradient descent while preserving computational simplicity. These methods can be generally formulated through the following iterative scheme:
\begin{equation}
	\bm{x}^{k+1} = \bm{x}^k - \beta_k \nabla f(\bm{x}^k) + \gamma_k \bm{m}^k.
	\label{it_acm}
\end{equation}
Here, $\bm{m}^k$ denotes the inertial term (also referred to as the momentum direction), which strategically incorporates information from previous iterates and gradients to enhance convergence. This concept draws inspiration from physics, where inertia characterizes an object's resistance to changes in its state of motion. Most accelerated first-order algorithms can be represented using formula \eqref{it_acm} by specifying appropriate values for the step size $\beta_k$, the inertial weight $\gamma_k$, and the inertial term $\bm{m}^k$.
Here are some prominent examples:
\begin{itemize}
	\item Heavy-Ball method \cite{polyak1964some} (HB):
	      \begin{equation*}
		      \beta_k = \beta \in \bigg(0, \frac{1 - \gamma}{D}\bigg], \quad \gamma_k = \gamma \in [0, 1), \quad \bm{m}_{\text{HB}}^k = \bm{x}^k - \bm{x}^{k - 1}.
	      \end{equation*}
	\item Nesterov Accelerated Gradient \cite{nesterov2013introductory} (NAG):
	      \begin{equation*}
		      \beta_k \in \bigg(0, \frac{1}{D} \bigg], \quad \gamma_k = 1, \quad \bm{m}_{\text{NAG}}^k = (\tilde{\bm{x}}^k - \bm{x}^k) - \beta_k \big[\nabla f(\tilde{\bm{x}}^k) - \nabla f(\bm{x}^k)\big],
	      \end{equation*}
		  where
		  \begin{align*}
			&f(\bm{x}^{k+1}) \leq f(\tilde{\bm{x}}^k) - \frac{\beta_k}{2} \| f(\tilde{\bm{x}}^k) \|^2 \text{ is ensured by backtracking line search,} \\
			& \text{choose } \theta_k \in (0, 1) \text{ such that } \frac{(1 - \theta_k) \beta_k}{\theta_k^2} \leq \frac{\beta_{k-1}}{\theta_{k-1}^2}, \\
			&\tilde{\bm{x}}^k = \bm{x}^k + \frac{\theta_k (1 - \theta_{k-1})}{\theta_{k-1}}(\bm{x}^k - \bm{x}^{k-1}).
		\end{align*}
	\item Adaptive Gradient Algorithm \cite{duchi2011adaptive} (AdaGrad):
	      \begin{equation*}
		      \beta_k = \beta, \quad \gamma_k = \beta, \quad \bm{m}_{\text{AdaGrad}}^k = \bigg[ \mathrm{I}_n - \bigg( \sqrt{\Diag\big( \hat{\bm{h}}^k + \bm{\varepsilon} \big)} \bigg)^{-1} \bigg] \nabla f(\bm{x}^k),
	      \end{equation*}
		  where
		  \begin{equation*}
			\hat{\bm{h}}^k = \hat{\bm{h}}^{k-1} + \nabla f(\bm{x}^k) \odot \nabla f(\bm{x}^k).
		\end{equation*}
	\item Adaptive Moment Estimation \cite{kingma2014adam} (Adam):
	      \begin{equation*}
		      \beta_k = \beta, \quad \gamma_k = \beta,\quad \bm m_{\text{Adam}}^k = \nabla f(\bm{x}^k) - \bigg( \sqrt{\Diag\big( \hat{\bm{h}}^k + \bm{\varepsilon} \big)} \bigg)^{-1} \hat{\bm{g}}^k,
	      \end{equation*}
	      where
	      \begin{align*}
		       & \hat{\bm{g}}^k = \alpha_1 \hat{\bm{g}}^{k-1} + (1 - \alpha_1) \nabla f(\bm{x}^k),\quad 0 < \alpha_1 < 1,                          \\
		       & \hat{\bm{h}}^k = \alpha_2 \hat{\bm{h}}^{k-1} + (1 - \alpha_2) \nabla f(\bm{x}^k) \odot \nabla f(\bm{x}^k),\quad 0 < \alpha_2 < 1.
	      \end{align*}
\end{itemize}

The theoretical study of accelerated first-order methods continues to attract significant attention. 
Contemporary research has predominantly focused on two main directions: Polyak-type acceleration \cite{aujol2022convergence, saab2022adaptive, aujol2023convergence, jin2024adaptive} and Nesterov-type methods \cite{carmon2018accelerated, farazmand2020multiscale, attouch2022ravine, josz2023convergence}. 
Additionally, many works have explored the physical interpretations of both Polyak's method \cite{attouch2000heavy, francca2020conformal} and Nesterov's acceleration \cite{wibisono2016variational, betancourt2018symplectic} through Lagrangian and Hamiltonian formalisms. 

The analyses in these works heavily rely on the specific formulation of the inertial term $\bm{m}^k$. 
This raises the question of whether convergence results can be established without depending on a particular form of the inertial term $\bm{m}^k$. 
While considerable work has addressed step size selection strategies $\beta_k$ through methods such as backtracking line search and trust region approaches \cite{fletcher2000practical}, the inertial weight parameter $\gamma_k$ is typically fixed rather than adaptively determined. 
Since the inertial term $\bm{m}^k$ evolves at each iteration, there naturally exist important mathematical relationships between $\gamma_k$ and $\bm{m}^k$ that warrant systematic investigation. 
The Adaptive Inertial Method (AIM) proposed in this paper addresses these two unexplored issues through a novel analytical framework that establishes rigorous connections between adaptive parameter selection and convergence properties.

\subsection*{Contributions}
\par We make the following key contributions:
\begin{itemize}
	\item[(1)] \textbf{Novel framework:} We introduce the Adaptive Inertial Method (AIM), a flexible framework that inherits and extends accelerated first-order methods through a customizable inertial term, establishing connections between various classical and modern optimization techniques.

	\item[(2)] \textbf{Theoretical guarantees:} We provide an elegant convergence analysis establishing a global convergence rate of $O(1/k)$ under mild conditions, requiring only convexity and local Lipschitz differentiability of the objective function, without additional assumptions on strong convexity.

	\item[(3)] \textbf{Adaptive parameter selection:} Our method enables fully adaptive parameter selection for the inertial term without manual tuning, making it particularly suitable for practical applications with flexible choices of inertial term.

	\item[(4)] \textbf{Customizable inertial terms:} Our method accommodates diverse choices of inertial terms beyond the standard heavy-ball momentum used in existing methods. This flexibility enables the incorporation of various acceleration effects into the optimization process, potentially facilitating the development of novel methods that would be challenging to analyze within classical theoretical frameworks.

	\item[(5)] \textbf{Novel Quasi-Newton formulation:} We derive the particular form of the inertial term that transforms AIM into a new Quasi-Newton method, demonstrating the framework's extensibility and theoretical versatility.

	\item[(6)] \textbf{Connection with regularized Newton:} Under specific circumstances, we prove that AIM coincides with the regularized Newton method, achieving an accelerated rate of $O(1/k^2)$ without requiring computationally expensive Hessian inversions, while only requiring one additional gradient evaluation.
\end{itemize}

\subsection*{Related works}
The convergence properties of first-order methods with acceleration have been extensively studied. 
For the heavy-ball method, Ghadimi and Lan \cite{ghadimi2015global} established its global convergence rate in terms of the average iterates, while Sun et al.\ \cite{sun2019non} later provided convergence guarantees for the last iterate. 
Drusvyatskiy et al.\ \cite{drusvyatskiy2018optimal}, inspired by the geometric approach of Bubeck et al.\ \cite{bubeck2015geometric}, developed a method that constructs and maintains a quadratic lower model of the objective function. Their approach guarantees that the model's minimum value converges to the true optimum at an optimal linear rate, providing both theoretical elegance and practical efficiency. 
Kim et al.\ \cite{kim2021optimizing} presented a systematic framework for optimizing step coefficients in first-order methods for smooth convex minimization. Their computationally efficient approach, based on performance estimation programming, achieves near-optimal worst-case gradient norm reduction, further advancing the theoretical foundations of accelerated optimization methods. 
In a parallel development, Xie et al.\ \cite{xie2022adaptive} proposed Adaptive Inertia (Adai), a novel optimization method that adaptively adjusts momentum to achieve faster convergence and better generalization by favoring flat minima, demonstrating superior performance over SGD and Adam variants across various datasets and models. More recent advances have focused on incorporating second-order information and adaptive techniques. 
Additionally, Zhang et al.\ \cite{zhang2022drsom} introduced DRSOM, which determines the coefficient of gradient and momentum direction by minimizing a quadratic approximation, effectively solving a two-dimensional linear system that incorporates Hessian approximation, theoretically achieving local quadratic convergence. 
Following this trend, Xie et al.\ \cite{xie2024adan} developed Adan, an innovative optimization algorithm for deep learning that combines adaptive learning rates with reformulated Nesterov momentum, enhancing convergence speed without the computational burden of traditional Nesterov acceleration. Their approach demonstrates superior performance across diverse tasks and models, supported by theoretical guarantees that align with optimal convergence rates. 
Most recently, Bao et al.\ \cite{bao2024accelerated} introduced a gradient restarted accelerated proximal gradient method, demonstrating global linear convergence for strongly convex composite optimization problems, with further validation through ODE model analysis for quadratic strongly convex objectives. Their work enhances the theoretical understanding of gradient restarting, showing it outperforms non-restarted Nesterov's accelerated gradient methods by achieving linear convergence rates, partially addressing an open question from prior literature.

\subsection*{Organization}
The remainder of this paper is organized as follows. Section~2 introduces the Adaptive Inertial Method, presenting its formulation, theoretical analysis, and convergence guarantees. In Section~3, we explore various options for constructing the inertial term, including physically-inspired inertial terms, Quasi-Newton formulations, and Hessian-gradient approaches. Section~4 presents comprehensive numerical experiments that demonstrate the effectiveness of AIM across different optimization problems, including logistic regression with $\mathcal{L}_2$ regularization and $\mathcal{L}_2$-$\mathcal{L}_p$ minimization problems. Finally, Section~5 summarizes our contributions and discusses potential directions for future research. Additional implementation details and experimental results are provided in the Appendix.

\subsection*{Notations}
\par Let $\mathbb{R}$ denote the set of real numbers and $\mathbb{N}$ the set of natural numbers. We represent vectors with bold lowercase letters and matrices with uppercase letters in normal format. For a vector $\bm{x}$, $\|\bm{x}\|$ denotes its Euclidean norm. Given a symmetric positive definite matrix $\mathrm{A}$, the $\mathrm{A}$-norm of a vector $\bm{x}$ is defined as $\|\bm{x}\|_{\mathrm{A}} = \sqrt{\bm{x}^\top \mathrm{A} \bm{x}}$. We use $\odot$ to represent the element-wise (Hadamard) product. For a vector $\bm{x}$, $\mathrm{Diag}(\bm{x})$ denotes the diagonal matrix whose diagonal elements correspond to the components of $\bm{x}$. Finally, $\mathrm{I}_n$ represents the $n \times n$ identity matrix.

\section{Adaptive Inertial Method}
\par We now introduce our framework for solving problem \eqref{prob_uc}. \vspace{3pt}

\noindent
\fbox{
	\parbox{0.95\textwidth}{
		\begin{subequations}
			Define $\bm{m}^k = \phi_k(\bm{x}^k, \bm{x}^{k-1}, \ldots, \bm{x}^0)$, where $\phi_k:\mathbb{R}^{n \times (k+1)} \rightarrow \mathbb{R}^n$ represents any nonzero mapping. For a given $\bm{x}^k$, set
			\begin{equation}
				\bm{x}^{k+1} = \bm{x}^k - \beta_k \nabla f(\bm{x}^k) + \gamma_k \bm{m}^k,
				\label{it_aim}
			\end{equation}
			where $\beta_k$ is a properly chosen parameter which satisfies
			\begin{equation}
				\beta_k (\bm{x}^k - \bm{x}^{k+1})^\top [\nabla f(\bm{x}^k) - \nabla f(\bm{x}^{k+1})] \leq \eta \|\bm{x}^k - \bm{x}^{k+1}\|^2, \quad \eta \in (0, 1),
				\label{cri_aim}
			\end{equation}
			and
			\begin{equation}
				\gamma_k = \mu_k \frac{(\bm{m}^k)^\top \nabla f(\bm{x}^k)}{\|\bm{m}^k\|^2} \beta_k,\quad \mu_k \in [0, 1).
				\label{gam_aim}
			\end{equation}
			\label{alg_aim}
		\end{subequations}
		\vspace{-9pt}
	}
}

\begin{remark}
	In this context, $\bm{m}^k$ represents the customizable inertial term, which serves as a critical component in our adaptive framework. We generalize its definition by allowing $\bm{m}^k$ to be constructed as a function of all preceding iterates. This flexibility enables us to incorporate a wide variety of momentum-like effects based on the optimization trajectory history.
\end{remark}

\begin{remark}
	Condition \eqref{cri_aim} can be satisfied when $\nabla f$ is locally Lipschitz continuous near $\bm{x}^k$. Various methods can be employed to ensure this condition, such as the Armijo rule. A detailed formulation will be provided in the pseudocode of our algorithm.
\end{remark}

\begin{remark}
	While the range restriction on $\mu_k \in [0, 1)$ means AIM cannot fully encompass existing methods such as the Heavy Ball method and Nesterov Accelerated Gradient method, this limitation is counterbalanced by the complete flexibility regarding the inertial term $\bm{m}^k$. This design choice allows any nonzero direction to be incorporated into the framework, significantly enhancing its adaptability and applicability across diverse optimization scenarios.
\end{remark}

To derive a more compact formulation, let us define the following:
\begin{equation}
	\Pi_k = \frac{\bm{m}^k (\bm{m}^k)^\top}{\|\bm{m}^k\|^2},
	\label{rm1_eq1} 
\end{equation}
and
\begin{equation}
	\mathrm{M}_k^{-1} = \mathrm{I}_n - \mu_k \Pi_k, \quad \mu_k \in [0, 1).
	\label{rm1_eq2}
\end{equation}
Note that $\Pi_k$ is the projection matrix onto $\bm{m}^k$ and thus $\mathrm{M}_k^{-1}$ is positive definite. It follows that
\begin{equation}
	\mathrm{M}_k = \mathrm{I}_n + \frac{\mu_k}{1 - \mu_k}\Pi_k, \quad \mu_k \in [0, 1),
	\label{rm1_eq3}
\end{equation}
and
\begin{equation}
	\mathrm{M}_k^{-1} \ \preceq \ \mathrm{I}_n \ \preceq \ \mathrm{M}_k.
	\label{rm1_eq4}
\end{equation}
Thus, we can rewrite equation \eqref{it_aim} as
\begin{equation}
	\bm{x}^{k+1} = \bm{x}^k - \beta_k \mathrm{M}_k^{-1} \nabla f(\bm{x}^k),
	\label{rm1_eq5}
\end{equation}
providing a more concise representation of the iteration.

\begin{remark}
	The iteration scheme \eqref{rm1_eq5} can also be expressed in an implicit form:
	\begin{equation*}
		\bm{x}^{k+1} = \bm{x}^k - \beta_k \mathrm{M}_k^{-1} \nabla f(\bm{x}^{k+1}),
	\end{equation*}
	which has an equivalent characterization
	\begin{equation*}
		\bm{x}^{k+1} \in \argmin\limits_{\bm{x} \in \mathbb{R}^n} \bigg\{ f(\bm{x}) + \frac{1}{2\beta_k}\|\bm{x} - \bm{x}^k\|_{\mathrm{M}_k}^2 \bigg\}.
	\end{equation*}
	This formulation reveals that AIM can be interpreted as a generalized proximal point algorithm (PPA), where the standard Euclidean distance metric is replaced by the $\mathrm{M}_k$ norm.
\end{remark}

\subsection{Theoretical Insight for Inertial Weight}
In this section, we provide the theoretical foundation underlying the selection of parameter $\gamma_k$ in equation \eqref{gam_aim}, which represents a pivotal analytical insight in the development of the Adaptive Inertial Method. 
Beginning with the iterative scheme defined in \eqref{it_aim}, we have
\begin{equation}
	\nabla f(\bm{x}^k) = \frac{1}{\beta_k} (\bm{x}^k - \bm{x}^{k+1} + \gamma_k \bm{m}^k).
	\label{equ_uc_1}
\end{equation}
By utilizing the convexity of $f$, we have
\begin{equation}
	f(\bm{x}) \geq f(\bm{x}^k) + (\bm{x} - \bm{x}^k)^\top \nabla f(\bm{x}^k), \quad \forall \bm{x} \in \mathbb{R}^n
	\label{equ_uc_cvx_1}
\end{equation}
and
\begin{equation}
	f(\bm{x}^k) \geq f(\bm{x}^{k+1}) + (\bm{x}^k - \bm{x}^{k+1})^\top \nabla f(\bm{x}^{k+1}).
	\label{equ_uc_cvx_2}
\end{equation}
Adding \eqref{equ_uc_cvx_1} and \eqref{equ_uc_cvx_2} together, we obtain
\begin{flalign*}
	f(\bm{x}) & \geq f(\bm{x}^{k+1}) + (\bm{x} - \bm{x}^k)^\top \nabla f(\bm{x}^k) + (\bm{x}^k - \bm{x}^{k+1})^\top \nabla f(\bm{x}^{k+1})                          \\
	          & = f(\bm{x}^{k+1}) + (\bm{x} - \bm{x}^{k+1})^\top \nabla f(\bm{x}^k) - (\bm{x}^k - \bm{x}^{k+1})^\top [\nabla f(\bm{x}^k) - \nabla f(\bm{x}^{k+1})].
\end{flalign*}
By combining with \eqref{cri_aim}, we arrive at
\begin{equation}
	f(\bm{x}) \geq f(\bm{x}^{k+1}) + (\bm{x} - \bm{x}^{k+1})^\top \nabla f(\bm{x}^k) - \frac{\eta}{\beta_k}\|\bm{x}^k - \bm{x}^{k+1}\|^2, \quad \forall \bm{x} \in \mathbb{R}^n.
	\label{equ_uc_2}
\end{equation}
Further incorporating \eqref{equ_uc_1}, we derive
\begin{equation*}
	f(\bm{x}) \geq f(\bm{x}^{k+1}) + \frac{1}{\beta_k} (\bm{x} - \bm{x}^{k+1})^\top (\bm{x}^k - \bm{x}^{k+1} + \gamma_k \bm{m}^k) - \frac{\eta}{\beta_k}\|\bm{x}^k - \bm{x}^{k+1}\|^2, \quad \forall \bm{x} \in \mathbb{R}^n.
\end{equation*}
Setting $\bm{x} = \bm{x}^k$ in the above inequality, we get
\begin{flalign*}
	f(\bm{x}^k) & \geq \ f(\bm{x}^{k+1}) + \frac{1 - \eta}{\beta_k}\|\bm{x}^k - \bm{x}^{k+1}\|^2 + \frac{1}{\beta_k} (\bm{x}^k - \bm{x}^{k+1})^\top \gamma_k \bm{m}^k                                                           \\
	            & \stackrel{\mathclap{\eqref{it_aim}}}{=} \ f(\bm{x}^{k+1}) + \frac{1 - \eta}{\beta_k}\|\bm{x}^k - \bm{x}^{k+1}\|^2 + \frac{1}{\beta_k} (\beta_k \nabla f(\bm{x}^k) - \gamma_k \bm{m}^k)^\top \gamma_k \bm{m}^k \\
	            & = \ f(\bm{x}^{k+1}) + \frac{1 - \eta}{\beta_k}\|\bm{x}^k - \bm{x}^{k+1}\|^2 + \bigg[ (\bm{m}^k)^\top \nabla f(\bm{x}^k) - \frac{\|\bm{m}^k\|^2}{\beta_k} \gamma_k \bigg] \gamma_k.
\end{flalign*}
We aim for the value of the objective function to consistently decrease, implying that $f(\bm{x}^{k+1}) < f(\bm{x}^k)$. A natural approach toward this goal involves establishing the following inequality
\begin{equation*}
	\bigg[ (\bm{m}^k)^\top \nabla f(\bm{x}^k) - \frac{\|\bm{m}^k\|^2}{\beta_k} \gamma_k \bigg] \gamma_k \geq 0,
\end{equation*}
which can be equivalently expressed as
\begin{equation*}
	\bigg[ \gamma_k - \frac{(\bm{m}^k)^\top \nabla f(\bm{x}^k)}{\|\bm{m}^k\|^2} \beta_k \bigg] \gamma_k \leq 0.
\end{equation*}
This inequality directly leads to the parameter selection rule presented in \eqref{gam_aim}.

\begin{remark}
	Zhang et al.\ \cite{zhang2024adam} proposed Adam-mini, which reduces memory requirements by optimizing the allocation of learning rate resources in Adam. The main idea involves partitioning parameters into blocks based on particular Hessian structure principles and assigning a single, well-calibrated learning rate to each parameter block. For the Adaptive Inertial Method, we can apply a similar strategy. We can transform $\gamma_k$ into a block diagonal matrix $\mathrm{\Gamma}_k$ with the following structure
	\begin{equation*}
		\mathrm{\Gamma}_k = \left(
		\begin{array}{cccc}
				\gamma_{k,1} \mathrm{I}_{n_1} & 0                             & \cdots & 0                             \\
				0                             & \gamma_{k,2} \mathrm{I}_{n_2} & \cdots & 0                             \\
				\vdots                        & \vdots                        & \ddots & \vdots                        \\
				0                             & 0                             & \cdots & \gamma_{k,T} \mathrm{I}_{n_T}
			\end{array}
		\right)
	\end{equation*}
	where $n_j$ are positive integers and $n_1 + n_2 + \cdots + n_T = n$. By applying similar analysis as presented above, we can derive $\gamma_{k,j}$ for each block $j = 1, 2, \cdots, T$ as
	\begin{equation*}
		\gamma_{k,j} = \mu_{k,j} \frac{(\bm{m}_j^k)^\top [\nabla f(\bm{x}^k)]_j}{\|\bm{m}_j^k\|^2} \beta_k,\quad \mu_{k,j} \in [0, 1).
	\end{equation*}
	Here $\bm{m}_j^k$ and $[\nabla f(\bm{x}^k)]_j$ denote the j-th segmented partition of $\bm{m}^k$ and $\nabla f(\bm{x}^k)$ respectively. This blockwise approach can be potentially effective for large-scale optimization problems.
\end{remark}

\subsection{Convergence Results}
The following lemmas formalize our key result regarding the Adaptive Inertial Method, which guarantees a strict reduction in the objective function value at each iteration and thereby establishes its fundamental descent property. Throughout this subsection, we denote the optimal solution of problem \eqref{prob_uc} by $\bm{x}^*$.
\begin{lemma}
	Let $\{\bm{x}^k\}$ be the sequence generated by Adaptive Inertial Method \eqref{alg_aim}. Then, the following result holds:
	\begin{equation}
		f(\bm{x}^{k+1}) \leq f(\bm{x}^k) - \frac{1 - \eta}{\beta_k}\|\bm{x}^k - \bm{x}^{k+1}\|_{\mathrm{M}_k}^2.
		\label{thm1_eq1}
	\end{equation}
	\label{thm1}
\end{lemma}

\begin{proof}
	\hspace{1pt} Revisiting \eqref{rm1_eq5}, we have
	\begin{equation}
		\nabla f(\bm{x}^k) = \frac{1}{\beta_k} \mathrm{M}_k (\bm{x}^k - \bm{x}^{k+1}).
		\label{thm1_eq2}
	\end{equation}
	Substituting \eqref{thm1_eq2} into \eqref{equ_uc_2} yeilds
	\begin{equation}
		f(\bm{x}) \geq f(\bm{x}^{k+1}) + \frac{1}{\beta_k} (\bm{x} - \bm{x}^{k+1})^\top \mathrm{M}_k (\bm{x}^k - \bm{x}^{k+1})  - \frac{\eta}{\beta_k}\|\bm{x}^k - \bm{x}^{k+1}\|^2.
		\label{thm1_eq3}
	\end{equation}
	Let $\bm{x} = \bm{x}^k$ in \eqref{thm1_eq3}, we get
	\begin{equation*}
		f(\bm{x}^{k+1}) \leq f(\bm{x}^k) - \frac{1}{\beta_k} \|\bm{x}^k - \bm{x}^{k+1}\|_{\mathrm{M}_k}^2 + \frac{\eta}{\beta_k}\|\bm{x}^k - \bm{x}^{k+1}\|^2.
		\label{thm1_eq4}
	\end{equation*}
	Apply \eqref{rm1_eq4} and we obtain \eqref{thm1_eq1}.
\end{proof}

\begin{remark}
	A notable refinement of our method can be achieved by relaxing condition \eqref{cri_aim} to
	\begin{equation}
		\beta_k (\bm{x}^k - \bm{x}^{k+1})^\top [\nabla f(\bm{x}^k) - \nabla f(\bm{x}^{k+1})] \leq \eta \|\bm{x}^k - \bm{x}^{k+1}\|^2_{\mathrm{M}_k}.
	\end{equation}
	This modification expands the feasible range for selecting $\beta_k$ values while preserving all theoretical convergence properties established in our analysis. The enhanced flexibility enables more aggressive step sizes in practice, potentially accelerating convergence without sacrificing algorithmic stability.
\end{remark}

\begin{lemma}
	Let $\{\bm{x}^k\}$ be the sequence generated by Adaptive Inertial Method \eqref{alg_aim}. Then, the following result holds:
	\begin{equation}
		\|\bm{x}^{k+1} - \bm{x}^*\|_{\mathrm{M}_k}^2 \leq \|\bm{x}^k - \bm{x}^*\|_{\mathrm{M}_k}^2 - (1 - 2\eta) \|\bm{x}^k - \bm{x}^{k+1}\|_{\mathrm{M}_k}^2 - 2\beta_k[f(\bm{x}^{k+1}) - f(\bm{x}^*)].
		\label{thm2_eq1}
	\end{equation}
	\label{thm2}
\end{lemma}

\begin{proof}
	\hspace{1pt} Take $\bm{x} = \bm{x}^*$ in \eqref{thm1_eq3}, we have
	\begin{equation*}
		f(\bm{x}^*) \geq f(\bm{x}^{k+1}) + \frac{1}{\beta_k} (\bm{x}^* - \bm{x}^{k+1})^\top \mathrm{M}_k (\bm{x}^k - \bm{x}^{k+1})  - \frac{\eta}{\beta_k}\|\bm{x}^k - \bm{x}^{k+1}\|^2.
		\label{thm2_eq2}
	\end{equation*}
	Utilizing property \eqref{rm1_eq4}, we obtain
	\begin{equation}
		(\bm{x}^k - \bm{x}^*)^\top \mathrm{M}_k (\bm{x}^k - \bm{x}^{k+1}) \geq (1 - \eta) \|\bm{x}^k - \bm{x}^{k+1}\|_{\mathrm{M}_k}^2 + \beta_k[f(\bm{x}^{k+1}) - f(\bm{x}^*)].
		\label{thm2_eq3}
	\end{equation}
	Notice that
	\begin{equation}
		2(\bm{x}^k - \bm{x}^*)^\top \mathrm{M}_k (\bm{x}^k - \bm{x}^{k+1}) = \|\bm{x}^k - \bm{x}^*\|_{\mathrm{M}_k}^2 + \|\bm{x}^k - \bm{x}^{k+1}\|_{\mathrm{M}_k}^2 - \|\bm{x}^{k+1} - \bm{x}^*\|_{\mathrm{M}_k}^2.
		\label{thm2_eq4}
	\end{equation}
	By combining \eqref{thm2_eq3} and \eqref{thm2_eq4}, we arrive at \eqref{thm2_eq1}.
\end{proof}

Examining Lemma~\ref{thm2}, we can make an observation: when $0 < \eta \leq 0.5$, the coefficient $(1 - 2\eta)$ becomes non-negative. Consequently, inequality \eqref{thm2_eq1} implies that $\|\bm{x}^{k+1} - \bm{x}^*\|_{\mathrm{M}_k} < \|\bm{x}^k - \bm{x}^*\|_{\mathrm{M}_k}$ for all $k \geq 0$ when $\bm{x}^{k+1} \neq \bm{x}^k$. This demonstrates that under this parameter regime, the Adaptive Inertial Method exhibits a contractive property with respect to the $\mathrm{M}_k$-norm.

Having established this foundation, we now proceed to analyze the convergence rate of the Adaptive Inertial Method. For analytical tractability, we adopt the approach of Beck and Teboulle \cite{beck2009fast} and assume a constant step size $\beta_k \equiv \beta > 0$ throughout the iterations in \eqref{alg_aim}. This simplification allows us to derive explicit bounds on the convergence behavior while maintaining the essential characteristics of the method.
\begin{lemma}
	Let $\{\bm x^k\}$ be the sequence generated by Adaptive Inertial Method \eqref{alg_aim}. Then, the following result holds:
	\begin{equation}
		2k\beta[f(\bm{x}^k) - f(\bm{x}^*)] \leq \|\bm{x}^0 - \bm{x}^*\|_{\mathrm{M}_k}^2 - \sum\limits_{j=0}^{k-1} [(1 - 2\eta) + 2j(1 - \eta)]\|\bm{x}^j - \bm{x}^{j+1}\|_{\mathrm{M}_k}^2.
		\label{thm3_eq1}
	\end{equation}
\end{lemma}

\begin{proof}
	From \eqref{thm2_eq1}, we have
	\begin{equation*}
		2\beta[f(\bm{x}^*) - f(\bm{x}^{k+1})] \geq \|\bm{x}^{k+1} - \bm{x}^*\|_{\mathrm{M}_k}^2 - \|\bm{x}_k - \bm{x}^*\|_{\mathrm{M}_k}^2 + (1 - 2\eta) \|\bm{x}^k - \bm{x}^{k+1}\|_{\mathrm{M}_k}^2,
		\label{thm3_eq2}
	\end{equation*}
	which, when summed over $j = 0, \dots, k - 1$, yields:
	\begin{equation}
		2\beta \bigg[k f(\bm{x}^*) - \sum\limits_{j=0}^{k-1} f(\bm{x}^{j+1}) \bigg] \geq -\|\bm{x}^0 - \bm{x}^*\|_{\mathrm{M}_k}^2 + \sum\limits_{j=0}^{k-1}(1-2\eta)\|\bm{x}^j - \bm{x}^{j+1}\|_{\mathrm{M}_k}^2.
		\label{thm3_eq3}
	\end{equation}
	Next, using \eqref{thm1_eq1}, we derive
	\begin{equation*}
		2\beta j[f(\bm{x}^j) - f(\bm{x}^{j+1})] \geq 2j(1 - \eta) \|\bm{x}^j - \bm{x}^{j+1}\|_{\mathrm{M}_k}^2,
		\label{thm3_eq4}
	\end{equation*}
	which further implies
	\begin{equation*}
		2\beta[jf(\bm{x}^j) - (j+1)f(\bm{x}^{j+1}) + f(\bm{x}^{j+1})] \geq 2j(1 - \eta) \|\bm{x}^j - \bm{x}^{j+1}\|_{\mathrm{M}_k}^2.
		\label{thm3_eq5}
	\end{equation*}
	Summing the above inequality over $j = 0, \dots, k - 1$, we obtain
	\begin{equation}
		2\beta\bigg[-k f(\bm{x}^k) + \sum\limits_{j=1}^{k-1} f(\bm{x}^{j+1}) \bigg] \geq \sum\limits_{j=1}^{k-1} 2j(1 - \eta) \|\bm{x}^j - \bm{x}^{j+1}\|_{\mathrm{M}_k}^2.
		\label{thm3_eq6}
	\end{equation}
	Finally, adding \eqref{thm3_eq3} and \eqref{thm3_eq6}, we arrive at \eqref{thm3_eq1}.
\end{proof}

When $\eta \in (0.5, 1)$, the coefficient $(1-2\eta)$ in the expression $(1-2\eta) + 2j(1-\eta)$ is negative. However, since $(1-\eta)$ remains positive for all $\eta \in (0, 1)$, this expression eventually becomes non-negative as $j$ increases. We formally define the threshold index by
\begin{equation}
	q(\eta) = \min \{ j \in \mathbb{N} \mid (1 - 2\eta) + 2j(1 - \eta) \geq 0 \}.
	\label{qeta}
\end{equation}
Note that for any $\eta \in (0.5, 1)$, $q(\eta)$ is a well-defined positive integer. This critical index allows us to establish the following convergence result.
\begin{theorem}
	Let $\{\bm{x}^k\}$ be the sequence generated by Adaptive inertia Method \eqref{alg_aim}. Then for $\forall \eta \in (0, 1)$, the following result holds:
	\begin{equation}
		f(\bm{x}^k) - f(\bm{x}^*) \leq \frac{\|\bm{x}^0 - \bm{x}^*\|_{\mathrm{M}_k}^2 + C}{2k\beta},
		\label{thm4_eq1}
	\end{equation}
	where
	\begin{equation*}
		C = - \sum\limits_{j=1}^{q(\eta) - 1} [(1 - 2\eta) + 2j(1 - \eta)]\|\bm{x}^j - \bm{x}^{j+1}\|_{\mathrm{M}_k}^2,
	\end{equation*}
	and $q(\eta)$ is defined in \eqref{qeta}.
	\label{thm4}
\end{theorem}
Theorem~\ref{thm4} establishes that the Adaptive Inertial Method achieves a global convergence rate of $O(1/k)$, which aligns with the standard rate for first-order methods applied to convex optimization problems. However, this convergence rate can be improved under specific conditions. In particular, when certain inertial terms are selected or when the objective function $f$ possesses special properties, the method can attain superior convergence rates such as superlinear or quadratic convergence. These enhanced convergence properties and specific conditions under which they arise will be demonstrated in the following section, where we explore various options for constructing effective inertial terms.

\section{Inertial Term Options}
The selection of an appropriate inertial term $\bm{m}^k$ is crucial for the performance and convergence properties of our proposed method. There exists a wide spectrum of possibilities for constructing effective inertial terms, each with distinct theoretical properties and practical implications. The flexibility of our framework allows us to incorporate insights from various optimization paradigms, resulting in different algorithmic behaviors. In this section, we explore several principled approaches to designing inertial terms, drawing inspiration from classical mechanics, quasi-Newton methods, and looking ahead techniques. These choices not only influence the theoretical convergence rates but also significantly impact the practical performance on different classes of optimization problems. By carefully selecting the inertial term, we can adapt our method to exploit problem-specific structures and characteristics.

\subsection{Physical Informed Inertia}
\label{subsec:phy_inert}
Drawing inspiration from classical mechanics, we can derive inertial terms that capture the essential dynamics of the optimization process. Two particularly effective choices, grounded in physical principles, are:
\begin{equation*}
	\bm{m}_v^k = \bm{x}^k - \bm{x}^{k-1} \quad\text{and}\quad \bm{m}_a^k = \nabla f(\bm{x}^k) - \nabla f(\bm{x}^{k-1}).
\end{equation*}
If we interpret \eqref{it_aim} as a discretized dynamical system, with $\bm x^k$ representing the coordinates, then $\bm m_v^k$ and $\bm m_a^k$ can be understood as velocity and acceleration, respectively. The classical heavy ball method \cite{polyak1964some} is inspired by the motion of a ball on a surface, where the ball's state of motion at any given time is influenced by its previous state. To account for this, a momentum term, $\bm{m}_v^k$, is introduced to accelerate the convergence of the algorithm. Our physical interpretation follows a similar rationale. While the classical descent direction is traditionally aligned with the negative gradient direction, incorporating the previous state or considering the effect of “inertial” requires an adjustment to this descent direction. By examining \eqref{rm1_eq2}, it becomes evident that when $\mu_k = 1$, the descent direction is orthogonal to the inertial term, effectively negating its influence. For $\mu_k \in (0,1)$, the descent direction retains a component of the inertial term. Furthermore, compared to $\bm{m}_v^k$, $\bm{m}_a^k$ demonstrates superior numerical performance, leading to a more pronounced acceleration effect.

\begin{figure}[h]
	\centering
	\begin{minipage}{0.45\textwidth}
		\centering
		\begin{tikzpicture}[scale=0.95, baseline=(current bounding box.center)]
			\coordinate (x_k) at (0,0);
			\coordinate (x_km1) at (-2,2);
			\coordinate (g_k) at (-2,0.5);
			\coordinate (n_g_k) at (2,-0.5);
			\coordinate (m_k) at (2,-2);
			\coordinate (g_k_proj) at (1.25,-1.25);
			\coordinate (g_k_proj_mid) at (0.625,-0.625);
			\coordinate (x_kp1) at (2,2/11);

			\draw[->, thick, -{Stealth[length=4mm, width=1.5mm]}] (x_km1) -- (x_k) node[draw=none] {};
			\draw[->, thick, orange, -{Stealth[length=4mm, width=1.5mm]}] (x_k) -- (m_k) node[right, xshift=0cm, yshift=0.1cm, draw=none] {$\bm{m}_v^k$};
			\draw[->, thick, blue, -{Stealth[length=4mm, width=1.5mm]}] (x_k) -- (g_k) node[right, above, draw=none] {$\nabla f(\bm{x}^k)$};
			\draw[dashed, thick, blue, -{Stealth[length=4mm, width=1.5mm]}] (x_k) -- (n_g_k) node[draw=none] {};
			\draw[dashed, thick, cyan] (n_g_k) -- (g_k_proj);
			\draw[dashed, thick] (n_g_k) -- (g_k_proj_mid);
			\draw[->, thick, -{Stealth[length=4mm, width=1.5mm]}] (x_k) -- (x_kp1);

			\draw pic [draw,cyan,angle radius=2mm] {right angle = n_g_k--g_k_proj--x_k};

			\filldraw[black] (x_k) circle (2pt) node[above right, draw=none] {$\bm{x}^k$};
			\filldraw[black] (x_km1) circle (2pt) node[left, draw=none] {$\bm{x}^{k-1}$};
			\filldraw[black] (x_kp1) circle (2pt) node[above right, draw=none] {$\bm{x}^{k+1}$};
		\end{tikzpicture}
		\caption{Choose $\bm{m}^k = \bm{m}_v^k$ and $\mu_k = 0.5$}
	\end{minipage}
	\hfill
	\begin{minipage}{0.45\textwidth}
		\centering
		\begin{tikzpicture}[scale=0.95, baseline=(current bounding box.center)]
			\coordinate (x_k) at (0,0);
			\coordinate (x_km1) at (-2,2);
			\coordinate (g_k) at (-2,0.5);
			\coordinate (n_g_k) at (2,-0.5);
			\coordinate (g_km1) at (-1.5,-1.5);
			\coordinate (m_k) at (-0.5,2);
			\coordinate (n_m_k) at (0.5,-2);
			\coordinate (g_k_proj) at (4/17,-16/17);
			\coordinate (g_k_proj_mid) at (2/17,-8/17);
			\coordinate (x_kp1) at (2,1/36);

			\draw[->, thick, -{Stealth[length=4mm, width=1.5mm]}] (x_km1) -- (x_k) node[draw=none] {};
			\draw[->, thick, orange, -{Stealth[length=4mm, width=1.5mm]}] (x_k) -- (m_k) node[right, xshift=0.1cm, yshift=0cm, draw=none] {$\bm{m}_a^k$};
			\draw[dashed, thick, orange] (x_k) -- (n_m_k);
			\draw[->, thick, blue, -{Stealth[length=4mm, width=1.5mm]}] (x_k) -- (g_km1) node[right, xshift=-0.9cm, yshift=-0.25cm, draw=none] {$\nabla f(\bm{x}^{k-1})$};
			\draw[->, thick, blue, -{Stealth[length=4mm, width=1.5mm]}] (x_k) -- (g_k) node[right, above, draw=none] {$\nabla f(\bm{x}^k)$};
			\draw[dashed, thick, blue, -{Stealth[length=4mm, width=1.5mm]}] (x_k) -- (n_g_k) node[draw=none] {};
			\draw[dashed, thick, cyan] (n_g_k) -- (g_k_proj);
			\draw[dashed, thick] (n_g_k) -- (g_k_proj_mid);
			\draw[->, thick, -{Stealth[length=4mm, width=1.5mm]}] (x_k) -- (x_kp1);

			\draw pic [draw,cyan,angle radius=2mm] {right angle = n_g_k--g_k_proj--n_m_k};

			\filldraw[black] (x_k) circle (2pt) node[above right, draw=none] {$\bm{x}^k$};
			\filldraw[black] (x_km1) circle (2pt) node[left, draw=none] {$\bm{x}^{k-1}$};
			\filldraw[black] (x_kp1) circle (2pt) node[above right, draw=none] {$\bm{x}^{k+1}$};
		\end{tikzpicture}
		\vspace{0.6mm}
		\caption{Choose $\bm{m}^k = \bm{m}_a^k$ and $\mu_k = 0.5$}
	\end{minipage}
\end{figure}

\subsection{Quasi-Newton Inertia}
\label{subsec:qn_inert}
\par Quasi-Newton methods \cite{broyden1967quasi} are powerful optimization techniques that approximate the Hessian matrix without explicitly computing second derivatives. These methods construct an approximation of the Hessian (or its inverse) using gradient information accumulated during the optimization process. In our framework, we can incorporate Quasi-Newton principles to design an effective inertial term.
\par We aim to construct $\mathrm{M}_k$, whose definition can be found in \eqref{rm1_eq3}, such that it serves as an appropriate approximation of the Hessian matrix at $\bm{x}^k$. This approach is particularly valuable when the objective function's curvature information significantly impacts convergence behavior. To achieve this, we require $\mathrm{M}_k$ to satisfy the secant equation, which is the fundamental Quasi-Newton condition:
\begin{equation}
	\mathrm{M}_k (\bm{x}^k - \bm{x}^{k-1}) = \alpha_k [\nabla f(\bm{x}^k) - \nabla f(\bm{x}^{k-1})],\quad \alpha_k > 0.
	\label{qn_eq1}
\end{equation}
Define $\bm{s}^{k-1} = \bm{x}^k - \bm{x}^{k-1}$ and $\bm{y}^{k-1} = \nabla f(\bm{x}^k) - \nabla f(\bm{x}^{k-1})$. When we substitute these definitions into equation \eqref{qn_eq1}, we get
\begin{equation*}
	\bigg( \mathrm{I}_n + \frac{\mu_k}{1 - \mu_k}\frac{\bm{m}^k (\bm{m}^k)^\top}{\|\bm{m}^k\|^2} \bigg) \bm{s}^{k-1} = \alpha_k \bm{y}^{k-1}.
\end{equation*}
After rearrangement, we have
\begin{equation*}
	\frac{\mu_k}{1 - \mu_k}\frac{(\bm{m}^k)^\top \bm{s}^{k-1}}{\|\bm{m}^k\|^2} \bm{m}^k = \alpha_k \bm{y}^{k-1} - \bm{s}^{k-1}.
\end{equation*}
This relationship suggests the appropriate choice for $\bm{m}^k$ should be
\begin{equation}
	\bm{m}^k = \alpha_k \bm{y}^{k-1} - \bm{s}^{k-1},
	\label{qn_eq2}
\end{equation}
with the constraint
\begin{equation*}
	\frac{\mu_k}{1 - \mu_k}\frac{(\bm{m}^k)^\top \bm{s}^{k-1}}{\|\bm{m}^k\|^2} = 1.
\end{equation*}
For $\mu_k \in [0, 1)$ to satisfy this constraint, we must have $(\bm{m}^k)^T \bm{s}^{k-1} > 0$, which implies
\begin{equation}
	\|\bm{s}^{k-1}\|^2 < \alpha_k (\bm{s}^{k-1})^\top \bm{y}^{k-1}.
\end{equation}
This curvature condition, commonly known as the secant equation positivity condition in quasi-Newton literature, implies that $f$ exhibits local strong convexity. Such a property is fundamental for ensuring the positive definiteness of Hessian approximations. Based on this requirement, we can establish the following lower bound for $\alpha_k$:
\begin{equation}
	\alpha_k > \frac{\|\bm{s}^{k-1}\|^2}{(\bm{s}^{k-1})^\top \bm{y}^{k-1}}.
	\label{alpha_k}
\end{equation}
Given $\alpha_k$ that satisfies \eqref{alpha_k}, we can determine the value of $\mu_k$ by
\begin{equation}
	\mu_k = \frac{\|\bm{m}^k\|^2}{\alpha_k (\bm{m}^k)^\top \bm{y}^{k-1}}.
\end{equation}
Substituting this expression into our update formula in \eqref{it_aim} yields the Quasi-Newton variant of AIM:
\begin{equation}
	\bm{x}^{k+1} = \bm{x}^k - \beta_k \left( \mathrm{I}_n - \frac{\bm{m}^k (\bm{m}^k)^\top}{\alpha_k (\bm{m}^k)^\top \bm{y}^{k-1}} \right) \nabla f(\bm{x}^k).
\end{equation}

\subsection{Hessian-gradient Inertia}
\label{subsec:hg_inert}
The concept of ``looking ahead" in optimization methods offers valuable intuition for constructing effective inertial terms. Nesterov's acceleration technique \cite{nesterov1983method} exemplifies this principle by incorporating future gradient information into the current update step. This insight motivates us to explore more sophisticated directional information that captures the local curvature of the objective function.
A natural extension involves incorporating second-order information through the Hessian-gradient product $\nabla^2 f(\bm{x}^k) \nabla f(\bm{x}^k)$. This product can be efficiently approximated using finite differences:
\begin{equation*}
	\nabla^2 f(\bm{x}^k) \nabla f(\bm{x}^k) \approx \frac{\nabla f(\bm{x}^k) - \nabla f(\bm{x}^k - \varepsilon \nabla f(\bm{x}^k))}{\varepsilon}
\end{equation*}
where $\varepsilon$ is a small positive parameter. The point $\tilde{\bm{x}}^k = \bm{x}^k - \varepsilon \nabla f(\bm{x}^k)$ represents a tentative step in the gradient direction, providing curvature information without requiring explicit computation of the Hessian matrix. As $\varepsilon \rightarrow 0$, this approximation converges to the exact Hessian-gradient product, capturing the directional derivative of the gradient in its steepest descent direction.
This mathematical relationship suggests that the Hessian-gradient product serves as an effective inertial term that naturally incorporates local curvature information. Such an approach provides a more sophisticated acceleration mechanism compared to traditional momentum methods, potentially leading to faster convergence in regions with varying curvature.

Let $\bm{g}^k = \nabla f(\bm{x}^k)$ and $\mathrm{H}_k = \nabla^2 f(\bm{x}^k)$ denote the gradient and Hessian at $\bm{x}^k$, respectively. When we choose $\bm{m}^k = \mathrm{H}_k \bm{g}^k$, the iteration \eqref{it_aim} becomes
\begin{equation}
	\bm{x}^{k+1} = \bm{x}^k - \beta_k \bm{g}^k + \beta_k \mu_k \frac{(\bm{g}^k)^\top \mathrm{H}_k \bm{g}^k}{(\bm{g}^k)^\top \mathrm{H}_k^2 \bm{g}^k} \mathrm{H}_k \bm{g}^k.
\end{equation}
Define
\begin{equation*}
	r_k = \frac{(\bm{g}^k)^\top \mathrm{H}_k^2 \bm{g}^k}{(\bm{g}^k)^\top \mathrm{H}_k \bm{g}^k},
\end{equation*}
which satisfies $\lambda_{\min}(\mathrm{H}_k) \leq r_k \leq \lambda_{\max}(\mathrm{H}_k)$. Our update formula then simplifies to
\begin{equation}
	\bm{x}^{k+1} = \bm{x}^k - \beta_k \bigg( \mathrm{I}_n - \frac{\mu_k}{r_k} \mathrm{H}_k \bigg) \bm{g}^k.
\end{equation}
Let $\kappa_k = \mathrm{cond}(\mathrm{H}_k)$ denote the condition number of $\mathrm{H}_k$. When $\mu_k < 1 / \kappa_k$, we have $\|\mu_k / r_k \mathrm{H}_k\| < 1$, enabling us to express the inverse as a convergent series
\begin{equation*}
	\bigg( \mathrm{I}_n - \frac{\mu_k}{r_k} \mathrm{H}_k \bigg)^{-1} = \sum_{j=0}^{\infty} \bigg(\frac{\mu_k}{r_k} \mathrm{H}_k \bigg)^j.
\end{equation*}
For sufficiently small values of $\mu_k$, we can effectively approximate this expression using only the first two terms
\begin{equation*}
	\bigg( \mathrm{I}_n - \frac{\mu_k}{r_k} \mathrm{H}_k \bigg)^{-1} \approx \mathrm{I}_n + \frac{\mu_k}{r_k} \mathrm{H}_k.
\end{equation*}
This approximation bears a strong resemblance to the regularized Newton method \cite{mishchenko2023regularized}, which follows the iteration scheme
\begin{equation}
	\bm{x}^{k+1} = \bm{x}^k - \big( \theta_k \mathrm{I}_n + \mathrm{H}_k \big)^{-1} \bm{g}^k,
\end{equation}
where $\theta_k = \sqrt{L \|\bm{g}^k\|}$. This formulation operates under the assumption that the Hessian of $f$ exhibits $2L$-Lipschitz continuity, meaning that for any vectors $\bm{x}$ and $\bm{y}$ in $\mathbb{R}^n$, the inequality $\| \nabla^2 f(\bm{x}) - \nabla^2 f(\bm{y}) \| \leq 2L \|\bm{x} - \bm{y}\|$ holds.
The regularized Newton method exhibits a global convergence rate of $O(1/k^2)$. Let us now examine the connection between the Adaptive Inertial Method and the regularized Newton method. We begin by equating the update steps at the same $\bm{x}^k$ of the two methods:
\begin{equation}
	\beta_k \bigg( \mathrm{I}_n - \frac{\mu_k}{r_k} \mathrm{H}_k \bigg) \bm{g}^k = \big( \theta_k \mathrm{I}_n + \mathrm{H}_k \big)^{-1} \bm{g}^k,
	\label{sec3.3_eq1}
\end{equation}
which is equivalent to
\begin{equation}
	\bigg[ \bigg( \theta_k - \frac{1}{\beta_k} \bigg) \mathrm{I}_n + \mathrm{H}_k - \theta_k \frac{\mu_k}{r_k} \mathrm{H}_k - \frac{\mu_k}{r_k} \mathrm{H}_k^2 \bigg] \bm{g}^k = 0.
\end{equation}
\begin{theorem}
	For each iteration $k$, suppose that there exists a scalar $a_k > 0$ such that $\mathrm{H}_k^2 = a_k \mathrm{H}_k$ (i.e., $\mathrm{H}_k$ is proportional to an idempotent matrix). If we select the parameters
	\begin{equation*}
		\beta_k = \frac{1}{\theta_k} \quad\text{and}\quad \mu_k = \frac{1}{1 + \theta_k / r_k},
	\end{equation*}
	then the Adaptive Inertial Method with inertial term $\bm{m}^k = \mathrm{H}_k \bm{g}^k$ becomes algebraically equivalent to the regularized Newton method when both methods are initialized at the same point.
\end{theorem}

\begin{remark}
	A notable case arises when the Hessian matrix $\mathrm{H}_k$ is rank one, which can be expressed as $\mathrm{H}_k = \bm{h}_k \bm{h}_k^T$ for some nonzero vector $\bm{h}_k$. In this case, $\mathrm{H}_k$ satisfies the condition $\mathrm{H}_k^2 = a_k \mathrm{H}_k$, where $a_k = \|\bm{h}_k\|^2 > 0$. This property enables the exact equivalence between our method and regularized Newton in specific scenarios.
\end{remark}

\begin{theorem}
	For each iteration k, let the characteristic polynomial of the Hessian matrix $\mathrm{H}_k$ be denoted as $p_k(t) = a_{k,0} + a_{k,1}t + \cdots + a_{k,n-1}t^{n-1} + t^n$, where $n > 1$. Notice that $p_k(\mathrm{H}_k) = a_{k,0} \mathrm{I}_n + a_{k,1} \mathrm{H}_k + \cdots + a_{k,n-1} \mathrm{H}_k^{n-1} + \mathrm{H}_k^n = 0$. Define a polynomial $q_k(t) = b_{k,0} + b_{k,1}t + \cdots + b_{k,n-1}t^{n-1}$ of degree $n-1$. If we choose the inertial term $\bm{m}^k = q_k(\mathrm{H}_k) \bm{g}^k$ and set
	\begin{equation*}
		r_k = \frac{(\bm{g}^k)^\top q_k(\mathrm{H}_k)^2 \bm{g}^k}{(\bm{g}^k)^\top q_k(\mathrm{H}_k) \bm{g}^k},
	\end{equation*}
	with coefficients defined by the recursive relationships:
	\begin{align*}
		b_{k,n-1} & = 1,                                              \\
		b_{k,n-2} & = a_{k,n-1} - \theta_k b_{k,n-1},                 \\
		          & \ \vdots                                          \\
		b_{k,1}   & = a_{k,2} - \theta_k b_{k,2},                     \\
		b_{k,0}   & = a_{k,1} - \theta_k b_{k,1} + \frac{r_k}{\mu_k},
	\end{align*}
	and step size parameter:
	\begin{equation*}
		\beta_k = \frac{1}{\theta_k + \frac{\mu_k}{r_k}(a_{k,0} - \theta_k b_{k,0})},
	\end{equation*}
	then the Adaptive Inertial Method coincides with the regularized Newton method, provided both methods are initialized at the same starting point.
	\label{thm3.3}
\end{theorem}

\begin{proof}
	Similar to equation \eqref{sec3.3_eq1}, we equate the update steps of the two methods
	\begin{equation}
		\beta_k \bigg( \mathrm{I}_n - \frac{\mu_k}{r_k} q_k(\mathrm{H}_k) \bigg) \bm{g}^k = \big( \theta_k \mathrm{I}_n + \mathrm{H}_k \big)^{-1} \bm{g}^k,
	\end{equation}
	which yields
	\begin{equation}
		\bigg[ \frac{\mu_k}{r_k} \mathrm{H}_k q_k(\mathrm{H}_k) + \theta_k \frac{\mu_k}{r_k} q_k(\mathrm{H}_k) - \mathrm{H}_k - \bigg( \theta_k - \frac{1}{\beta_k} \bigg) \mathrm{I}_n \bigg] \bm{g}^k = \bigg( \sum\limits_{i=0}^n c_{k,i} \mathrm{H}_k^i \bigg) \bm{g}^k = 0,
	\end{equation}
	where the coefficients are related as follows
	\begin{align*}
		c_{k,n} & = \frac{r_k}{\mu_k} b_{k,n-1},                                                 \\
		c_{k,i} & = \frac{r_k}{\mu_k} ( b_{k,i-1} + \theta_k b_{k,i}), \quad i = n-1, \ldots, 2, \\
		c_{k,1} & = \frac{r_k}{\mu_k} ( b_{k,0} + \theta_k b_{k,1}) - 1,                         \\
		c_{k,0} & = \frac{r_k}{\mu_k} \theta_k b_{k,0} - \theta_k + \frac{1}{\beta_k}.
	\end{align*}
	To establish proportional coefficients, we let
	\begin{equation}
		c_{k,i} = \frac{r_k}{\mu_k} a_{k,i}, \quad i = 0, \ldots, n
	\end{equation}
	with $a_{k,n}=1$. This formulation leads directly to the desired coefficient relationships stated in the theorem, thus completing the proof.
\end{proof}

\begin{remark}
	For practical computation, we can set $b_{k,0}=0$ in Theorem \ref{thm3.3}. This leads to the relation
	\begin{equation}
		\frac{r_k}{\mu_k} = \theta_k b_{k,1} - a_{k,1},
	\end{equation}
	which eliminates the need to separately calculate $r_k$ and find $\mu_k$. Instead, we can directly work with their ratio as a unified parameter.
\end{remark}

\section{Numerical Experiments}
In this section, we present comprehensive numerical experiments to evaluate the performance of AIM across a diverse range of optimization problems. Our experimental evaluation consists of three main components: (1) logistic regression with $\mathcal{L}_2$ regularization, a standard convex problem in machine learning; (2) $\mathcal{L}_2$-$\mathcal{L}_p$ minimization, which introduces nonconvexity when $0 < p < 1$.

\subsection{Experimental Setup}
We implemented AIM and all benchmark algorithms in Python, enabling fair and consistent comparisons. All experiments were conducted on a workstation running macOS with an Apple M3 Max processor. The complete implementation of our algorithms and experimental code is available at \url{https://github.com/harmoke/AIM}.

For comparative analysis, we considered several established optimization methods as benchmark al gorithms, including Gradient Descent method (GD), heavy-Ball method (HB), Nesterov Accelerated Gradient (NAG), Adam and DRSOM.
For all algorithms, we conducted comprehensive hyperparameter tuning to ensure optimal performance, establishing a fair benchmark for comparison. We measured algorithm efficiency by tracking both computational time and iteration count required to reach a first-order stationary point satisfying the stopping criterion:
\begin{equation*}
	\| \nabla f(\bm{x}^k) \| \leq 10^{-6}.
\end{equation*}
For all experiments, we initialized algorithms with $\bm{x}^0 = \bm{0}$ to ensure consistent comparison. Detailed experimental configurations and comprehensive results for each problem class are presented in the following subsections, where we analyze both convergence behavior and computational efficiency.

\subsection{Logistic Regression with $\mathcal{L}_2$ Regularization}
We first consider the logistic regression problem with $\mathcal{L}_2$ regularization:
\begin{equation}
	\min_{\bm{x} \in \mathbb{R}^n} -\frac{1}{n} \sum_{i=1}^{n} \Big( b_i \log \big( \sigma(\bm{a}_i^\top \bm{x}) \big) + (1 - b_i) \log \big( 1 - \sigma(\bm{a}_i^\top \bm{x}) \big) \Big) + \frac{\lambda}{2} \|\bm{x}\|_2^2,
\end{equation}
where $\sigma(z) = \frac{1}{1+e^{-z}}$ is the sigmoid function, $\bm{a}_i \in \mathbb{R}^n$ represent the feature vectors, and $b_i \in \{0,1\}$ are the corresponding binary labels. For our empirical evaluation, we employed the ``real-sim" (72309 samples, 20958 features) datasets from the LIBSVM repository \cite{chang2011libsvm}. We examined regularization parameters $\lambda \in \{10^{-5}, 10^{-6}, 10^{-7}\}$ to assess algorithm robustness across different degrees of regularization strength.

We present comprehensive experimental results in Table \ref{tab:logistic_regression}, which summarizes the performance of all algorithms across three different regularization strengths. As the regularization parameter $\lambda$ decreases from $10^{-5}$ to $10^{-7}$, we observe a systematic increase in computational effort across all algorithms. This behavior aligns with established optimization theory, as weaker regularization leads to poorer conditioning of the objective function, resulting in a more challenging optimization landscape. Notably, the magnitude of this performance degradation varies substantially across methods, providing valuable insights into their inherent robustness to ill-conditioning.

\begin{table}[ht]
	\centering
	\caption{Complete results on Logistic Regression with \(\mathcal{L}_2\) Regularization}
	\label{tab:logistic_regression}
	\begin{tabular}{c|cc|cc|cc}
		\toprule
		\multirow{2}{*}{Algorithm} & \multicolumn{2}{c|}{\(\lambda = 10^{-5}\)} & \multicolumn{2}{c|}{\(\lambda = 10^{-6}\)} & \multicolumn{2}{c}{\(\lambda = 10^{-7}\)}                             \\
		\cmidrule{2-7}
		                           & \(k\)                                      & time/s                                     & \(k\)                                     & time/s  & \(k\) & time/s  \\
		\midrule
		GD                         & 123                                        & 0.843245                                   & 509                                       & 3.56752 & 2672  & 19.4235 \\
		HB                         & 169                                        & 1.16347                                    & 264                                       & 1.8618  & 532   & 3.8013  \\
		NAG                        & 82                                         & 5.23994                                    & 222                                       & 14.2364 & 816   & 72.2274 \\
		AdaGrad                    & 685                                        & 4.66478                                    & 2262                                      & 16.2666 & 5197  & 37.2778 \\
		Adam                       & 166                                        & 1.1914                                     & 177                                       & 1.28382 & 452   & 3.2897  \\
		DRSOM                      & 35                                         & 0.820609                                   & 87                                        & 2.07868 & 251   & 5.99746 \\
		AIM\_v                     & 79                                         & 1.35853                                    & 219                                       & 3.29396 & 1209  & 16.787  \\
		AIM\_a                     & 52                                         & 0.714274                                   & 157                                       & 2.23386 & 295   & 4.17481 \\
		AIM\_QN                    & 66                                         & 0.912178                                   & 198                                       & 2.88666 & 309   & 4.66182 \\
		AIM\_Hg                    & 42                                         & 0.748909                                   & 89                                        & 1.63089 & 238   & 4.52331 \\
		\bottomrule
	\end{tabular}
\end{table}

Traditional first-order methods (GD, HB) demonstrate reasonable performance for stronger regularization but deteriorate substantially as $\lambda$ decreases. Among established adaptive methods, DRSOM consistently outperforms other benchmarks, requiring significantly fewer iterations across all regularization settings, with Adam following as a close second. NAG achieves competitive iteration counts but incurs substantially higher computational overhead per iteration, particularly for weaker regularization, suggesting potential implementation inefficiencies or intrinsically higher per-iteration complexity. AdaGrad generally required a higher number of iterations compared to the other tested methods, suggesting it may face certain challenges when applied to problems with this specific structure and conditioning characteristics.

The various AIM implementations demonstrate exceptional performance across all experimental configurations. Among these, AIM with Hessian-gradient inertia (AIM\_Hg) consistently achieves superior convergence, requiring the fewest iterations while maintaining computational efficiency comparable to DRSOM. The acceleration-based inertial term (AIM\_a) systematically outperforms the velocity-based variant (AIM\_v). The Quasi-Newton formulation (AIM\_QN) exhibits robust performance across all problem dimensions and sparsity levels, offering a well-balanced compromise between convergence speed and computational overhead. These results empirically validate our theoretical framework, demonstrating that appropriate inertial term selection can significantly enhance optimization performance.

\begin{figure}[ht]
	\centering
	\begin{subfigure}[b]{0.48\textwidth}
		\centering
		\includegraphics[width=\textwidth]{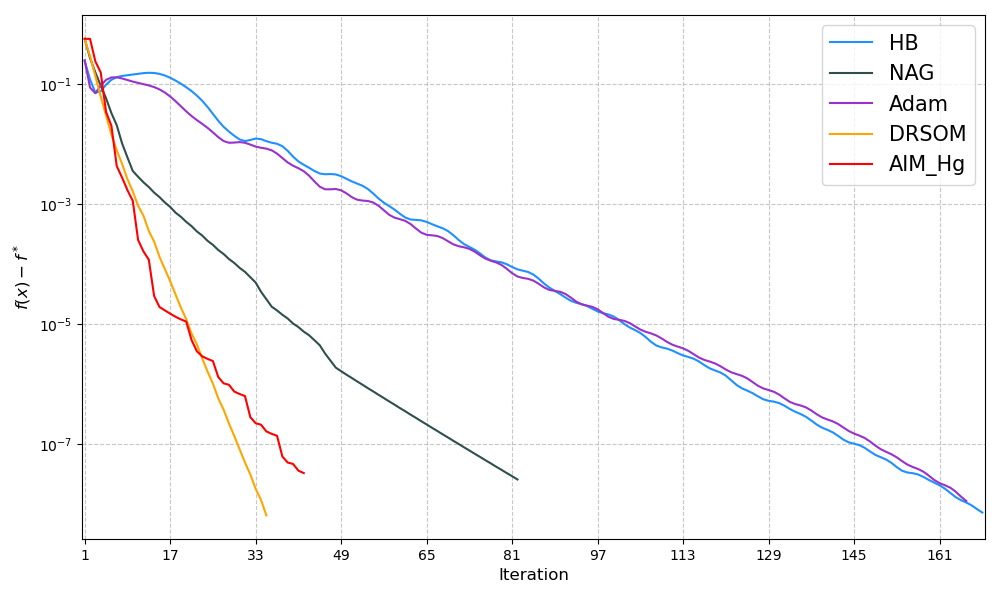}
		\caption{$\lambda=10^{-5}$}
		\label{fig:logistic_1e-5}
	\end{subfigure}
	\hfill
	\begin{subfigure}[b]{0.48\textwidth}
		\centering
		\includegraphics[width=\textwidth]{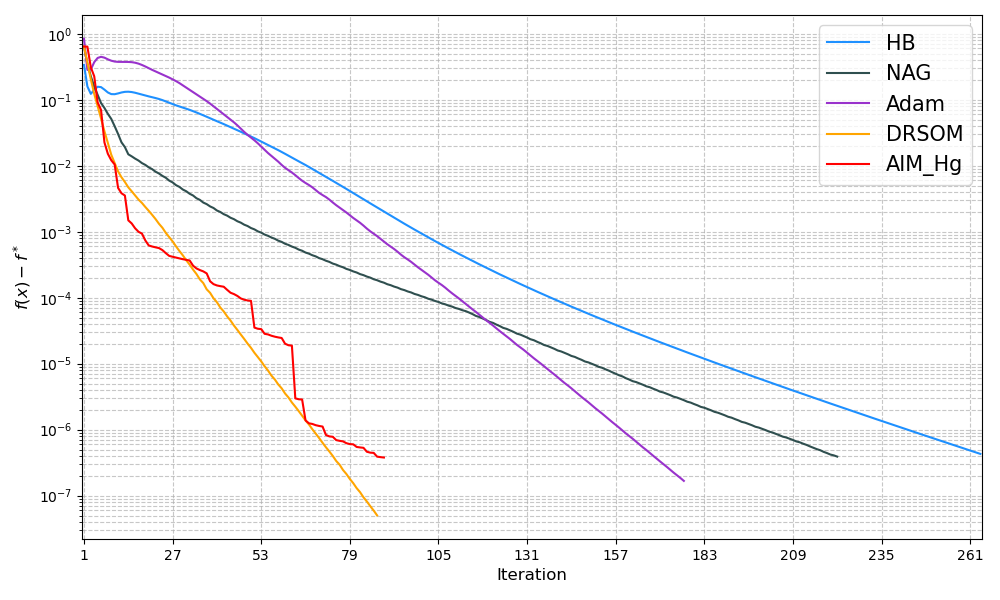}
		\caption{$\lambda=10^{-6}$}
		\label{fig:logistic_1e-6}
	\end{subfigure}

	\vspace{0.5cm}
	\begin{subfigure}[b]{0.48\textwidth}
		\centering
		\includegraphics[width=\textwidth]{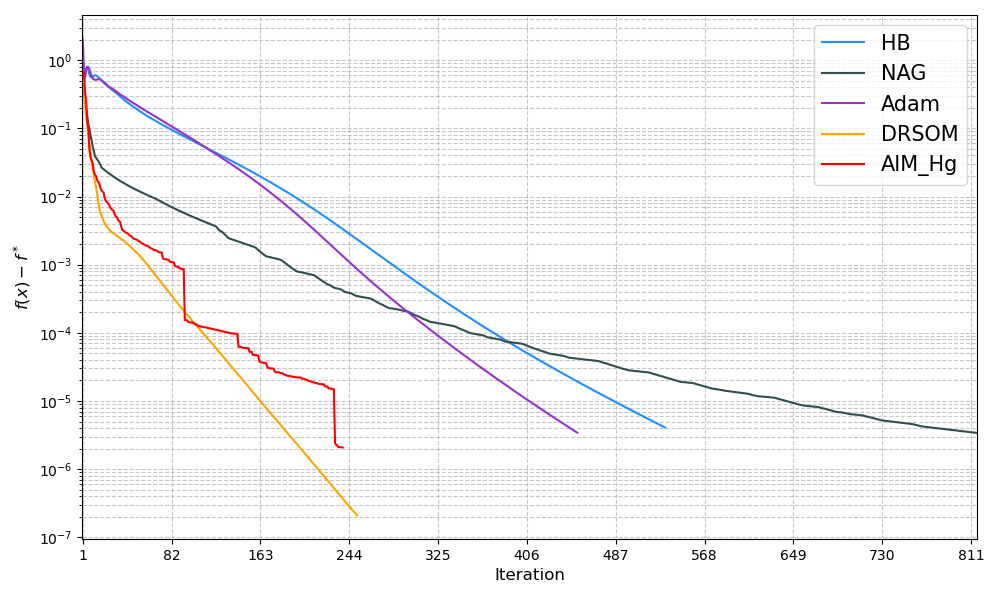}
		\caption{$\lambda=10^{-7}$}
		\label{fig:logistic_1e-7}
	\end{subfigure}
	\caption{Convergence behavior of AIM and selected benchmark algorithms for logistic regression with $\mathcal{L}_2$ regularization under varying regularization strengths. The convergence trajectory of AIM with Hessian-gradient inertia (AIM\_Hg) exhibits minor oscillations but closely approximates that of DRSOM, which achieves the $O(1/k^2)$ convergence rate for convex problems.}
	\label{fig:logistic_regression}
\end{figure}

\subsection{$\mathcal{L}_2$-$\mathcal{L}_p$ Minimization}
We next evaluate AIM on the $\mathcal{L}_2$-$\mathcal{L}_p$ minimization problem, formulated as:
\begin{equation}
	\min_{\bm{x} \in \mathbb{R}^n} \frac{1}{2}\|\mathrm{A}\bm{x} - \bm{b}\|_2^2 + \lambda\|\bm{x}\|_p^p,
\end{equation}
where $\mathrm{A} \in \mathbb{R}^{m \times n}$, $\bm{b} \in \mathbb{R}^m$, and $\lambda > 0$ is a regularization parameter. This problem class is particularly interesting as its convexity properties depend critically on the value of $p$. When $p = 2$, the problem reduces to standard ridge regression (a quadratic program). For $p = 1$, we obtain the well-known Lasso regression, which promotes sparsity while maintaining convexity. When $0 < p < 1$, the problem becomes nonconvex but provides enhanced sparsity promotion, making it valuable for applications in compressed sensing and sparse signal recovery.

To address the nonsmoothness of $\|\cdot\|_p$ when $0 < p \leq 1$, we implement the smoothing technique proposed by Chen \cite{chen2012smoothing}. This approach replaces the $\mathcal{L}_p$ norm with a differentiable approximation:
\begin{equation*}
	\|\bm{x}\|_p^p \approx \sum_{i=1}^{n} s(x_i, \varepsilon)^p,
\end{equation*}
where $\varepsilon > 0$ is a small smoothing parameter and $s(z, \epsilon)$ is a piecewise function that approximates $|z|$:
\begin{equation}
	s(z, \epsilon) = \begin{cases}
		|z|                                           & \text{if}\ |z| > \varepsilon,    \\
		\frac{z^2}{2\epsilon} + \frac{\varepsilon}{2} & \text{if}\ |z| \leq \varepsilon.
	\end{cases}
\end{equation}

For our experiments, we followed the data generation procedure described in Zhang et al. \cite{zhang2022drsom}. We generated random datasets with varying dimensions $m$ and $n$, where the elements of matrix $\mathrm{A}$ were sampled from a standard normal distribution $\mathcal{N}(0, 1)$ with controlled sparsity levels of $r = 15\%$ and $25\%$. To construct the ground truth sparse vector $\bm{v} \in \mathbb{R}^n$, we applied the following mixed distribution:
\begin{equation*}
	v_i \sim
	\begin{cases}
		0                                   & \text{with probability } q = 0.5, \\
		\mathcal{N}\big(0, \frac{1}{n}\big) & \text{otherwise}.
	\end{cases}
\end{equation*}
The response vector was then generated as $\bm{b} = \mathrm{A}\bm{v} + \bm{\delta}$, where $\bm{\delta}$ represents noise with components $\delta_i \sim \mathcal{N}(0, 1)$ for all $i$. We set the regularization parameter $\lambda$ to $\frac{1}{5} \|\mathrm{A}^\top \bm{b}\|_\infty$ and examined different norm parameters $p \in \{0.5, 1, 2\}$ with a smoothing parameter of $\varepsilon = 0.1$.

Tables \ref{tab:p=0.5}, \ref{tab:p=1}, and \ref{tab:p=2} present comprehensive results for $\mathcal{L}_2$-$\mathcal{L}_p$ minimization across different values of $p$, problem dimensions, and sparsity patterns. AIM with Hessian-gradient inertia (AIM\_Hg) demonstrates remarkably robust performance across all test configurations, consistently outperforming traditional first-order methods while remaining competitive with or surpassing DRSOM.

\begin{figure}[htbp]
	\centering
	\begin{subfigure}[b]{0.48\textwidth}
		\centering
		\includegraphics[width=\textwidth]{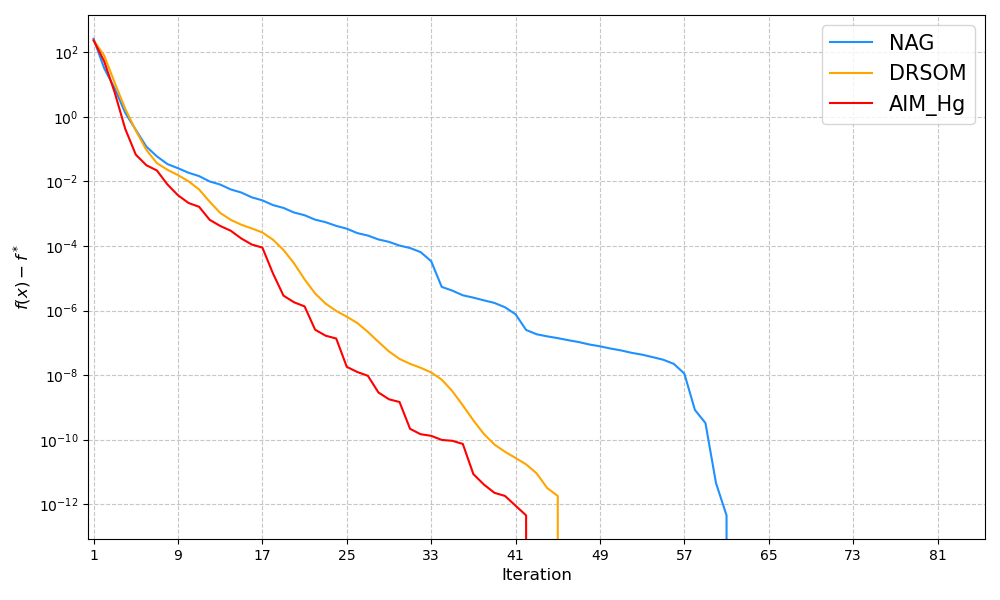}
		\caption{$p=0.5, (m,n)=(1000,1500), r=15\%$}
		\label{fig:p=0.5,r=0.15}
	\end{subfigure}
	\quad
	\begin{subfigure}[b]{0.48\textwidth}
		\centering
		\includegraphics[width=\textwidth]{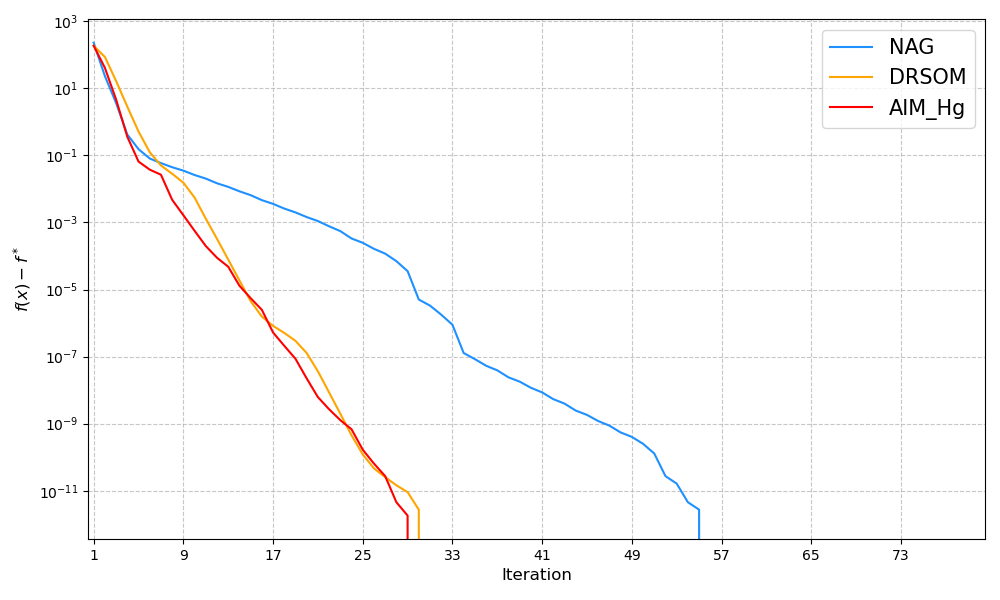}
		\caption{$p=0.5, (m,n)=(1000,1500), r=25\%$}
		\label{fig:p=0.5,r=0.25}
	\end{subfigure}

	\begin{subfigure}[b]{0.48\textwidth}
		\centering
		\includegraphics[width=\textwidth]{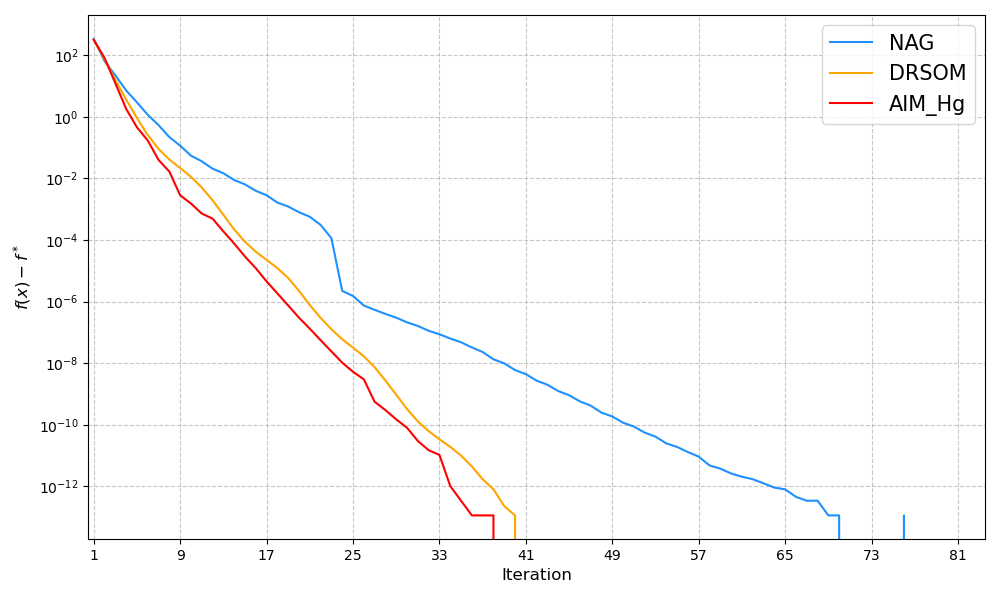}
		\caption{$p=1, (m,n)=(1000,1500), r=15\%$}
		\label{fig:p=1,r=0.15}
	\end{subfigure}
	\quad
	\begin{subfigure}[b]{0.48\textwidth}
		\centering
		\includegraphics[width=\textwidth]{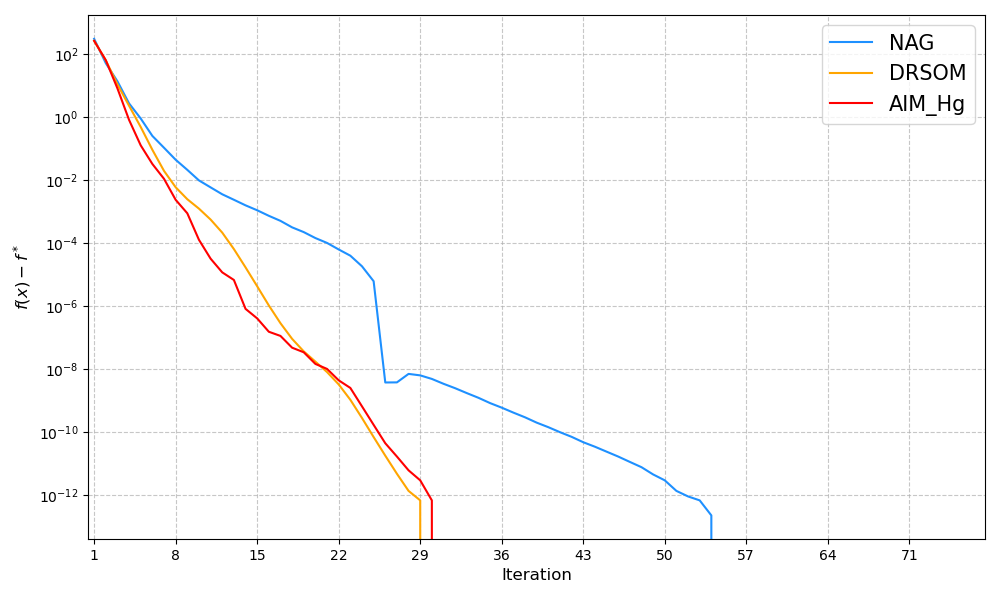}
		\caption{$p=1, (m,n)=(1000,1500), r=25\%$}
		\label{fig:p=1,r=0.25}
	\end{subfigure}

	\begin{subfigure}[b]{0.48\textwidth}
		\centering
		\includegraphics[width=\textwidth]{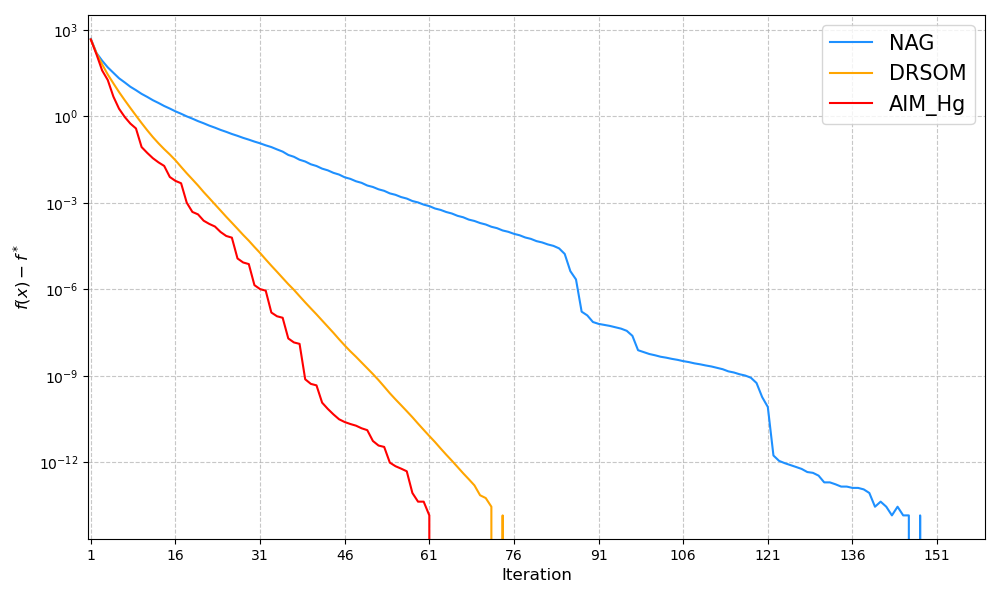}
		\caption{$p=2, (m,n)=(1000,1500), r=15\%$}
		\label{fig:p=2,r=0.15}
	\end{subfigure}
	\quad
	\begin{subfigure}[b]{0.48\textwidth}
		\centering
		\includegraphics[width=\textwidth]{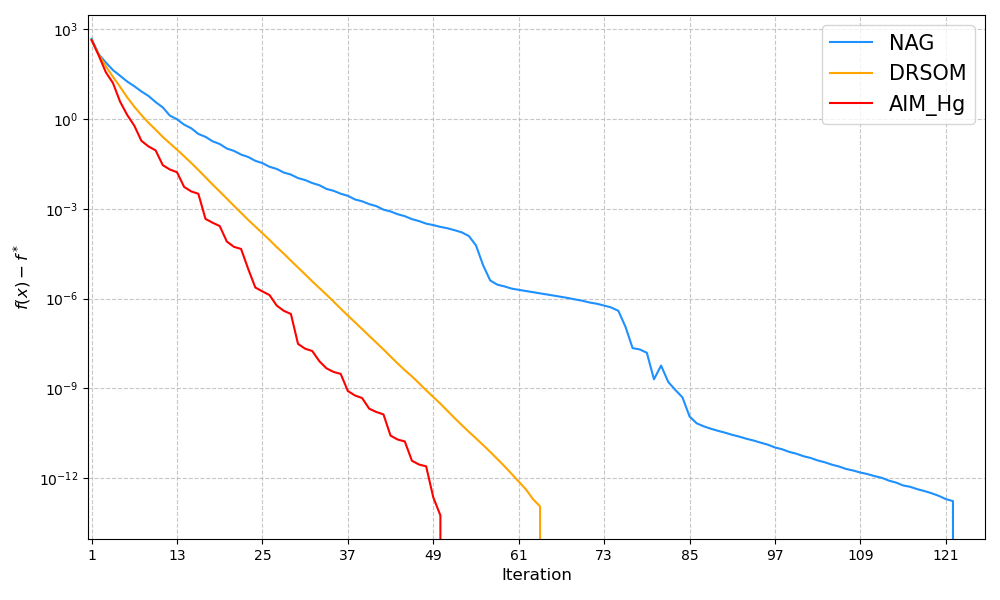}
		\caption{$p=2, (m,n)=(1000,1500), r=25\%$}
		\label{fig:p=2,r=0.25}
	\end{subfigure}

	\caption{Convergence behavior of AIM and selected benchmark algorithms for $\mathcal{L}_2-\mathcal{L}_p$ minimization. AIM with Hessian-gradient inertia (AIM\_Hg) consistently outperforms all other methods. While exhibiting slightly more oscillatory behavior than DRSOM, AIM\_Hg achieves steeper descent trajectories.}
	\label{fig:l2_lp}
\end{figure}

Notably, AIM\_Hg exhibits superior scalability as problem dimensions increase, with more modest growth in iteration counts compared to other methods. Its performance remains consistent across different sparsity levels (15\% vs 25\%), indicating robustness to problem structure variations. Among all AIM variants, the Hessian-gradient formulation consistently delivers the best results, confirming the theoretical advantages of incorporating local curvature information through the inertial term.

\section{Conclusions}
In this paper, we presented the Adaptive Inertial Method (AIM), a novel optimization framework that unifies and extends accelerated first-order methods through customizable inertial terms. Our work makes several significant contributions to the field of convex optimization. First, we developed a flexible framework that adaptively selects parameters without manual tuning, requiring only convexity and local Lipschitz differentiability of the objective function. Second, we proved a global convergence rate of $O(1/k)$ under mild conditions. Third, we demonstrated that under specific parameter choices, AIM coincides with the regularized Newton method, achieving an accelerated rate of $O(1/k^2)$ without requiring computationally expensive matrix inversions.

The flexibility of our framework allows for various inertial term constructions, including those inspired by physical principles, quasi-Newton methods, and Hessian-gradient products, each offering distinct advantages for different problem classes. Our numerical experiments confirm the theoretical properties of AIM and demonstrate its superior performance across diverse optimization scenarios compared to existing methods.

Several promising research directions emerge from this work. First, further theoretical analysis could establish tighter convergence rates for AIM under various conditions, particularly when specific inertial terms are employed. Second, extending AIM to non-convex optimization landscapes merits investigation, as many practical applications involve non-convex objective functions. Third, developing distributed and stochastic variants would enhance AIM's applicability to large-scale machine learning problems. Fourth, exploring connections between AIM and continuous dynamical systems could provide deeper insights into the acceleration mechanisms. Finally, designing problem-specific inertial terms for specialized domains such as deep learning, reinforcement learning, and scientific computing represents a fertile ground for future research.

\bibliographystyle{siamplain}
\bibliography{references}

\newpage
\appendix
\section{Example Pseudocode of AIM}
\begin{algorithm}[H]
	\caption{Adaptive Inertial Method (AIM) with Hessian-gradient Inertia}
	\label{alg:aim}
	\begin{algorithmic}[1]
		\setlength{\baselineskip}{15pt}
		\State \textbf{Input:}  \(\nabla f\)
		\State \textbf{Parameters:} \(\eta = 0.9\), \(\mu = 0.75\), \(\varepsilon = 10^{-3}\), \(\text{gtol} = 10^{-6}\), \(\text{mtol} = 10^{-8}\)
		\State \textbf{Initialization:} \(\bm{x}^0 = \bm{0}\), \(\beta_0 = 1.0\), \(k = 0\)
		\While{True}
		\If{\(\|\nabla f(\bm{x}^k)\| < \text{gtol}\)}
		\State \textbf{break}
		\EndIf
		\State \(\bm{m}^k = \Big(\nabla f(\bm{x}^k) - \nabla f \big(\bm{x}^k - \varepsilon \nabla f(\bm{x}^k) \big) \Big) \ / \ \varepsilon\)
		\If{\(\|\bm{m}^k\| < \text{mtol}\)}
		\State \(\bm{m}^k = \bm{0}\)
		\Else
		\State \(\bm{m}^k = \bm{m}^k \ / \ \|\bm{m}^k\| \)
		\EndIf
		\While{True}
		\State \(\bm{x}^{k+1} = \bm{x}^k - \beta_k \nabla f(\bm{x}^k) + \mu \beta_k (\bm{m}^k)^\top \nabla f(\bm{x}^k) \bm{m}^k\)
		\State \( r_k = \beta_k (\bm{x}^k - \bm{x}^{k+1})^\top (\nabla f(\bm{x}^k) - \nabla f(\bm{x}^{k+1})) \ / \ \|\bm{x}^k - \bm{x}^{k+1}\|_{\mathrm{M}_k}^2 \)
		\If{\(r_k > \eta\)}
		\State \( \beta_k = \beta_k \ / \ 1.5 * \min(1,\ 1 \ / \ r_k) \)
		\Else
		\State \textbf{break}
		\EndIf
		\EndWhile
		\If{\(r_k < 0.5\)}
		\State \( \beta_{k+1} = \beta_k * 2 \ / \ r_k \)
		\Else
		\State \(\beta_{k+1} = \beta_k\)
		\EndIf
		\State \(k = k + 1\)
		\EndWhile
		\State \Return \(\bm{x}^k\)
	\end{algorithmic}
\end{algorithm}

\begin{remark}
	For determining step sizes that satisfy condition \eqref{cri_aim}, we implement the computational strategy outlined in Algorithm \ref{alg:aim} (lines 14-27), following the approach of He et al.\ \cite{he2002improvements}. While the constants in this strategy can be adjusted based on specific problem characteristics, we maintained consistent parameter settings across all numerical experiments.
\end{remark}

\begin{sidewaystable}[ht]
	\section{Detailed Experimental Results}
	\centering
	\caption{Complete Results on $\mathcal{L}_2-\mathcal{L}_p$ Minimization for $p=0.5$}
	\label{tab:p=0.5}
	\resizebox{\textheight}{!}{
		\begin{tabular}{c cccc cccc cccc}
			\toprule
			        & \multicolumn{4}{c}{$(m,n)=(1000,500)$} & \multicolumn{4}{c}{$(m,n)=(1000,1000)$} & \multicolumn{4}{c}{$(m,n)=(1000,1500)$}                                                                                                                                                   \\
			\cmidrule{2-5} \cmidrule{6-9} \cmidrule{10-13}
			        & \multicolumn{2}{c}{$r=15\%$}           & \multicolumn{2}{c}{$r=25\%$}            & \multicolumn{2}{c}{$r=15\%$}            & \multicolumn{2}{c}{$r=25\%$} & \multicolumn{2}{c}{$r=15\%$} & \multicolumn{2}{c}{$r=25\%$}                                                      \\
			\cmidrule{2-3} \cmidrule{4-5} \cmidrule{6-7} \cmidrule{8-9} \cmidrule{10-11} \cmidrule{12-13}
			Method  & k                                      & time/s                                  & k                                       & time/s                       & k                            & time/s                       & k    & time/s   & k    & time/s   & k   & time/s   \\
			\midrule
			GD      & 119                                    & 0.30141                                 & 161                                     & 0.908201                     & 190                          & 1.62147                      & 126  & 0.546443 & 270  & 2.69129  & 164 & 1.6003   \\
			HB      & 362                                    & 1.3651                                  & 369                                     & 1.37096                      & 369                          & 2.02555                      & 375  & 2.31107  & 374  & 2.5371   & 379 & 3.50114  \\
			NAG     & 67                                     & 5.04304                                 & 64                                      & 4.54759                      & 135                          & 14.6337                      & 57   & 8.63799  & 85   & 13.8672  & 80  & 19.5722  \\
			AdaGrad & 350                                    & 1.07925                                 & 354                                     & 0.812473                     & 605                          & 3.21023                      & 745  & 3.67999  & 764  & 4.71473  & 572 & 4.12131  \\
			Adam    & 361                                    & 1.01731                                 & 368                                     & 0.812322                     & 3790                         & 17.6491                      & 2215 & 11.3415  & 3873 & 26.2299  & 375 & 2.73652  \\
			DRSOM   & 39                                     & 0.15768                                 & 33                                      & 0.12755                      & 46                           & 0.454836                     & 34   & 0.370155 & 55   & 0.616161 & 37  & 0.509491 \\
			AIM\_v  & 60                                     & 0.324762                                & 54                                      & 0.309137                     & 82                           & 0.768103                     & 58   & 0.738433 & 95   & 1.36166  & 64  & 0.851241 \\
			AIM\_a  & 58                                     & 0.323237                                & 51                                      & 0.17441                      & 88                           & 1.08104                      & 56   & 0.868122 & 124  & 1.93673  & 75  & 1.0193   \\
			AIM\_QN & 78                                     & 0.859208                                & 70                                      & 0.583572                     & 116                          & 1.34263                      & 69   & 1.1416   & 141  & 2.13134  & 87  & 1.22072  \\
			AIM\_Hg & 32                                     & 0.269338                                & 27                                      & 0.236947                     & 40                           & 0.448958                     & 28   & 0.763978 & 48   & 0.915107 & 37  & 0.868842 \\
			\bottomrule
		\end{tabular}
	}
\end{sidewaystable}

\begin{sidewaystable}[ht]
	\centering
	\caption{Complete Results on $\mathcal{L}_2-\mathcal{L}_p$ Minimization for $p=1$}
	\label{tab:p=1}
	\resizebox{\textheight}{!}{
		\begin{tabular}{c cccc cccc cccc}
			\toprule
			        & \multicolumn{4}{c}{$(m,n)=(1000,500)$} & \multicolumn{4}{c}{$(m,n)=(1000,1000)$} & \multicolumn{4}{c}{$(m,n)=(1000,1500)$}                                                                                                                                             \\
			\cmidrule{2-5} \cmidrule{6-9} \cmidrule{10-13}
			        & \multicolumn{2}{c}{r=15\%}             & \multicolumn{2}{c}{r=25\%}              & \multicolumn{2}{c}{r=15\%}              & \multicolumn{2}{c}{r=25\%} & \multicolumn{2}{c}{r=15\%} & \multicolumn{2}{c}{r=25\%}                                                      \\
			\cmidrule{2-3} \cmidrule{4-5} \cmidrule{6-7} \cmidrule{8-9} \cmidrule{10-11} \cmidrule{12-13}
			Method  & k                                      & time/s                                  & k                                       & time/s                     & k                          & time/s                     & k    & time/s   & k   & time/s   & k    & time/s   \\
			\midrule
			GD      & 81                                     & 0.115108                                & 80                                      & 0.199609                   & 167                        & 1.52548                    & 179  & 1.11336  & 179 & 1.82131  & 144  & 1.6629   \\
			HB      & 363                                    & 1.17214                                 & 370                                     & 1.72182                    & 369                        & 1.51987                    & 376  & 2.04679  & 375 & 2.7871   & 380  & 2.4446   \\
			NAG     & 73                                     & 4.90803                                 & 53                                      & 3.89539                    & 86                         & 9.83025                    & 128  & 14.0669  & 83  & 12.4255  & 77   & 13.0144  \\
			AdaGrad & 257                                    & 0.53429                                 & 354                                     & 0.72822                    & 584                        & 2.37499                    & 521  & 2.01797  & 637 & 4.18466  & 713  & 4.64413  \\
			Adam    & 362                                    & 0.879815                                & 368                                     & 2.10337                    & 3443                       & 15.848                     & 1608 & 6.80827  & 376 & 2.22447  & 2515 & 14.4353  \\
			DRSOM   & 31                                     & 0.129948                                & 31                                      & 0.143879                   & 41                         & 0.276894                   & 38   & 0.315482 & 46  & 0.467993 & 36   & 0.344985 \\
			AIM\_v  & 60                                     & 0.320659                                & 61                                      & 0.310611                   & 72                         & 0.616292                   & 62   & 0.62095  & 72  & 1.00381  & 61   & 0.778789 \\
			AIM\_a  & 53                                     & 0.313399                                & 55                                      & 0.233581                   & 64                         & 0.516485                   & 66   & 0.609043 & 85  & 1.51958  & 65   & 0.833325 \\
			AIM\_QN & 71                                     & 0.305318                                & 65                                      & 0.381454                   & 82                         & 0.753776                   & 81   & 0.69029  & 100 & 1.54455  & 85   & 1.12944  \\
			AIM\_Hg & 26                                     & 0.289671                                & 29                                      & 0.187835                   & 36                         & 0.365211                   & 30   & 0.389351 & 39  & 0.74752  & 33   & 0.502909 \\
			\bottomrule
		\end{tabular}
	}
\end{sidewaystable}

\begin{sidewaystable}[ht]
	\centering
	\caption{Complete Results on $\mathcal{L}_2-\mathcal{L}_p$ Minimization for $p=2$}
	\label{tab:p=2}
	\resizebox{\textheight}{!}{
		\begin{tabular}{c cccc cccc cccc}
			\toprule
			        & \multicolumn{4}{c}{$(m,n)=(1000,500)$} & \multicolumn{4}{c}{$(m,n)=(1000,1000)$} & \multicolumn{4}{c}{$(m,n)=(1000,1500)$}                                                                                                                                          \\
			\cmidrule{2-5} \cmidrule{6-9} \cmidrule{10-13}
			        & \multicolumn{2}{c}{r=15\%}             & \multicolumn{2}{c}{r=25\%}              & \multicolumn{2}{c}{r=15\%}              & \multicolumn{2}{c}{r=25\%} & \multicolumn{2}{c}{r=15\%} & \multicolumn{2}{c}{r=25\%}                                                   \\
			\cmidrule{2-3} \cmidrule{4-5} \cmidrule{6-7} \cmidrule{8-9} \cmidrule{10-11} \cmidrule{12-13}
			Method  & k                                      & time/s                                  & k                                       & time/s                     & k                          & time/s                     & k   & time/s   & k   & time/s  & k   & time/s   \\
			\midrule
			GD      & 127                                    & 0.627286                                & 183                                     & 0.64682                    & 337                        & 2.48032                    & 432 & 2.93249  & 582 & 4.40264 & 628 & 7.64102  \\
			HB      & 365                                    & 1.68929                                 & 372                                     & 1.21685                    & 371                        & 2.16005                    & 376 & 1.88266  & 424 & 2.52242 & 460 & 2.79461  \\
			NAG     & 64                                     & 4.16884                                 & 107                                     & 9.99522                    & 191                        & 21.5741                    & 176 & 20.0268  & 159 & 24.0464 & 126 & 20.6198  \\
			AdaGrad & 581                                    & 1.529                                   & 631                                     & 1.3828                     & 781                        & 3.77067                    & 487 & 2.33737  & 780 & 4.4354  & 726 & 4.36306  \\
			Adam    & 487                                    & 1.71674                                 & 437                                     & 1.11713                    & 366                        & 1.72003                    & 655 & 2.77245  & 378 & 2.11224 & 374 & 2.70752  \\
			DRSOM   & 39                                     & 0.168569                                & 42                                      & 0.173702                   & 65                         & 0.488232                   & 66  & 0.570909 & 76  & 0.91293 & 72  & 0.824639 \\
			AIM\_v  & 93                                     & 0.613957                                & 105                                     & 0.722503                   & 232                        & 2.00246                    & 229 & 1.90265  & 235 & 3.14365 & 193 & 2.3774   \\
			AIM\_a  & 67                                     & 0.4052                                  & 75                                      & 0.310385                   & 154                        & 1.55412                    & 170 & 1.42721  & 188 & 2.57264 & 199 & 2.8096   \\
			AIM\_QN & 74                                     & 0.371021                                & 77                                      & 0.386523                   & 165                        & 1.57105                    & 141 & 1.1852   & 219 & 2.86106 & 173 & 2.46134  \\
			AIM\_Hg & 36                                     & 0.461458                                & 38                                      & 0.220492                   & 53                         & 0.753512                   & 58  & 0.78714  & 62  & 1.17894 & 55  & 0.928901 \\
			\bottomrule
		\end{tabular}
	}
\end{sidewaystable}

\end{document}